\documentclass{amsart}[12pt]
\usepackage{amsmath, amsthm, amscd, amsfonts,}
\usepackage{graphicx,float}

\usepackage{amssymb}

\usepackage{pb-diagram}
\usepackage{lamsarrow}
\usepackage{pb-lams}

\newcommand{\ga}{\alpha}

\newcommand{\gk}{\kappa}

\newcommand{\gn}{\nu}

\newcommand{\gt}{\tau}

\newcommand{\func }{\mathord{:}}

\newcommand{\restricted}{\mathord{\restriction}}

\newcommand{\ordered}[1]{\ensuremath{\langle #1 \rangle}}

\newcommand{\ordof}[2]{\ensuremath{\ordered{ #1 \mid #2 }}}

\DeclareMathOperator{\len}{l}

\DeclareMathOperator{\crit}{crit}
\DeclareMathOperator{\dom}{dom}

\DeclareMathOperator{\Col}{Col}
\DeclareMathOperator{\Add}{Add}
\DeclareMathOperator{\C}{C}
\DeclareMathOperator{\supp}{supp}
\DeclareMathOperator{\Ult}{Ult}
\DeclareMathOperator{\limdir}{lim\ dir}
\DeclareMathOperator{\id}{id}

\DeclareMathOperator{\mc}{mc}

\newcommand{\Es}{{\ensuremath{\bar{E}}\/}}

\newcommand{\VS}{V^*}
\newcommand{\MS}{{M^*}}

\newcommand{\NS}{{N^*}}
\newcommand{\NSE}{{N^{*\Es}}}
\newcommand{\MSt}{{M^*_\gt}}
\newcommand{\Mt}{{M_\gt}}

\newcommand{\MStE}{{M^{* \Es}_\gt}}

\newcommand{\ME}{{M_{\Es}}}

\newcommand{\MSE}{{M^*_{\Es}}}

\def\k{\kappa}

\def\l{\lambda}

\def\a{\alpha}

\def\b{\beta}

\newtheorem{theorem}{Theorem}[section]
\newtheorem{lemma}[theorem]{Lemma}

\newtheorem{definition}[theorem]{Definition}

\newtheorem{remark}[theorem]{Remark}
\newtheorem{claim}[theorem]{Claim}
\numberwithin{equation}{section}

\usepackage{pb-diagram}

\def\l{\lambda}

\def\rmark{\mbox{$\rm\bf\rule{0.06em}{1.45ex}\kern-0.05em R$}}
\def\pmark{\mbox{$\rm\bf\rule{0.06em}{1.45ex}\kern-0.05em P$}}
\def\nmark{\mbox{$\rm\bf\rule{0.06em}{1.45ex}\kern-0.05em N$}}
\def\vdash{\mbox{$\rm\| \kern-0.13em -$}}

\begin{document}

\title[Killing GCH everywhere by a cofinality-preserving forcing notion]{Killing GCH everywhere by a cofinality-preserving forcing notion over a model of $GCH$ }

\author[Sy D. Friedman and M. Golshani]{Sy-David Friedman and Mohammad
  Golshani}

\thanks{The first author would like to thank the FWF (Austrian Science Fund) for
its support through Project P23316-N13.}

\thanks{The second author's research was in part supported by a grant from IPM (No. 91030417).}

\thanks{Both authors would like to thank Carmi Merimovich for his
reading of the manuscript and for his helpful suggestions. They also thank the referee of the paper for his useful comments and remarks.}
\maketitle




\begin{abstract}
Starting from large cardinals we construct a pair $V_1\subseteq V_2$ of models of $ZFC$ with the same cardinals and cofinalities such that $GCH$ holds in $V_1$ and fails everywhere in $V_2$.

\end{abstract}
\maketitle

\section{Introduction}

Easton's classical result showed that over any model of $GCH,$ one can
force any reasonable pattern of the power function $\lambda\mapsto
2^\lambda$ on the regular cardinals $\lambda$, preserving cardinals
and cofinalities. Subsequently, much work has been done on the
singular cardinal problem, whose aim is to characterize the patterns
of the power function on all cardinals, including the singular ones.
Typically in this work, large cardinals are used to obtain patterns of
power function behavior at singular cardinals after applying subtle
forcings which change cofinalities or even collapse cardinals.
This leads one to ask: Is it possible to obtain a failure of $GCH$
everywhere by forcing over a model of $GCH$ without changing
cofinalities? If so, can one have a fixed finite gap in the resulting
model, meaning that $2^\lambda=\lambda^{+n}$ for some finite $n>1$ for
all $\lambda$?

In this paper we prove the following theorem.
\begin{theorem}
Assume $GCH+$there exists a $(\kappa+4)-$strong cardinal $\kappa$. Then there is a pair $V_1\subseteq V_2$ of models of $ZFC$ such that:

$(a)$ $V_1$ and $V_2$ have the same cardinals and cofinalities,

$(b)$ $GCH$ holds in $V_1$,

$(c)$ $V_{2} \models `` \forall \lambda, 2^{\lambda} = \lambda^{+3}$''.
\end{theorem}

\begin{remark}
In fact it suffices to have a Mitchell increasing sequence of extenders of length $\kappa^{+}$, each of them $(\kappa+3)-$strong.
Thus the exact strength that we need for a fixed gap of 3 is a cardinal $\kappa$ with
$o(\kappa)=\kappa^{+3}+\kappa^{+}$. It is also easy to extend our
result to an arbitrary finite gap $n$ instead of $3.$
Then what we  need is a cardinal $\kappa$ with
$o(\kappa)=\kappa^{+n}+\kappa^{+}.$ We focus on the case $n=3$ as it
is typical of all cases $n\geq 3$ (the case $n=2$ is easier).
\end{remark}

The rest of this paper is devoted to the proof of this theorem. The
proof is based on the extender-based Radin forcing developed by
C. Merimovich in the papers [3], [4]. We try to make the proof
self-contained, thus we start with some preliminaries and facts from
these papers, suitably modified for our purposes.

We now summarise the modifications of [4] which are necessary to
achieve our result. In [4], one begins with a model $V^*$ with a
$(\kappa + 4)$-strong cardinal $\kappa$ and performs a (cofinality-preserving) reverse
Easton preparation, which forces $2^\alpha=\alpha^{+3}$ for
the first three successors of each inaccessible $\leq\kappa$. In
the resulting model $V = V^*[G]$ one can construct suitable
``guiding generics'' for later use, which are in fact generics
over a suitable inner model $M$ of $V$ which blow up the power sets
of the first three successors of $\kappa^{+3}$ and which collapse
the image $i(\kappa)$ of $\kappa$ to $\kappa^{+6}$, where $i:V\to M$
is a suitable elementary embedding. After this preparation, one
performs an extender-based Radin forcing with interleaved collapses,
using the guiding generics obtained through preparation. The result
is a model with gap $3$ everywhere below $\kappa$ (i.e., $2^\alpha=
\alpha^{+3}$ for all $\alpha<\kappa$). By truncating the universe
at $\kappa$, one obtains gap $3$ everywhere.

We would like to use a similar method, however we need to perform
a preparation which preserves the GCH below $\kappa$. Thus our
first step is to obtain a model $V = V^*[G]$ which only forces
$2^\alpha=\alpha^{+3}$ at the first three successors of $\kappa$
and adds no new subsets of $\kappa$. Extra work is now required to
show that in this model suitable guiding generics can be found to
carry out the second step of Merimovich's construction. The result
is again a model $V_2 = V[G][H]$ with gap $3$ everywhere below
$\kappa$ (keeping $\kappa$ inaccessible). We now form a model $V_1$
intermediate between $V[G]$ and $V_2$, essentially obtained by using
the ordinary Radin forcing with interleaved collapses (using the
collapsing part of the guiding generics). The model $V_1$ satisfies
GCH below $\kappa$ but has the same cofinalities below $\kappa$
as the model $V_2$. This is verified using a suitable projection
from Merimovich's extender-based Radin forcing with collapses into
the ordinary Radin forcing with collapses.

We should mention that obtaining models $V_1\subseteq V_2$ with
the same \emph{cardinals} (not the same cofinalities) and with the
GCH holding in $V_1$ but failing everywhere in $V_2$ is an easier
result, as then we only need guiding generics for Cohen forcings, not
for L\'evy colllapses, and the second step of the forcing can consist
of a cardinal-preserving (but of course not cofinality-preserving)
Radin forcing. But to preserve cofinalities or to obtain the gap $3$
behaviour of the power function it
appears that the methods of this paper are needed to handle the
necessary collapses.

\section{Extender Sequences}

Suppose $j: V^{*} \rightarrow M^{*} \supseteq V_{\lambda}^{*},
crit(j)=\kappa.$ Define an extender (with projections)

\begin{center}
$E(0)= \langle \langle E_{\alpha}(0): \alpha \in \emph{A} \rangle,
\langle \pi_{\beta, \alpha}: \beta, \alpha \in \emph{A}, \beta
\geq_{j} \alpha \rangle \rangle$
\end{center}

on $\kappa$ by:

\begin{itemize}
\item $\emph{A}=[\kappa, \lambda),$ \item $\forall \alpha \in \emph{A}, E_{\alpha}(0)$ is the $\kappa-$complete ultrafilter on $\kappa$ defined by

    \begin{center}
    $X \in E_{\alpha}(0) \Leftrightarrow \alpha \in j(X)$
    \end{center}
We write $E_\alpha(0)$ as $U_\alpha$.
    \item $\forall \alpha, \beta \in \emph{A}$
     \begin{center}
     $\beta \geq_{j} \alpha \Leftrightarrow \beta \geq \alpha$ and for some $ f \in$$ ^{\kappa} \kappa,$ $  j(f)(\beta)=\alpha$
     \end{center}
     \item $\beta \geq_{j} \alpha \Rightarrow \pi_{\beta, \alpha}: \kappa \rightarrow \kappa$ is such that $j(\pi_{\beta, \alpha})(\beta)=\alpha$

\end{itemize}

Let's recall the main properties of $E(0)$ (see [2])
\begin{enumerate}
\item $\langle \emph{A}, \leq_j \rangle$ is a $\k^+-$directed partial order,
\item $\forall \a, \k \leq_j \a,$ \item $U_\k$ is a normal measure on
  $\k$, \item $\forall \a, U_\a$ is a $P-$point ultrafilter over $\k,$
  i.e for any $f: \k \rightarrow \k$ there is $X \in U_\a$ such that
  $\forall \nu< \k, |X \cap f^{-1''}(\nu)|< \k,$ \item $\pi_{\b,
    \a}^{-1''}(X) \in U_\b \Leftrightarrow X \in U_\a,$ \item $\forall
  \a, \pi_{\a, \a}=id,$
\item $\forall \gamma \geq_j \b \geq_j \a$
  there is $X \in U_\gamma$ such that $\forall \nu \in X, \pi_{\gamma,
   \a}(\nu)=\pi_{\b, \a}(\pi_{\gamma, \b}(\nu)),
$ \item $\forall
  \gamma \geq_j \a, \b$ where $\a \neq \b$ there is $X \in U_\gamma$ such
  that $\forall \nu \in X, \pi_{\gamma, \a}(\nu) \neq \pi_{\gamma,
    \b}(\nu),$

Moreover the $\pi_{\alpha,\kappa}$'s can be chosen so that:

\item $\forall \b \geq_j \a, \forall \nu< \k, \pi_{\b,
    \k}(\nu)=\pi_{\a, \k}(\pi_{\b, \a}(\nu)),$
\item $\forall \a, \b,
  \forall \nu< \k, \pi_{\a, \k}(\nu)=\pi_{\b, \k}(\nu)$; we denote the
  latter by $\nu^0.$
\end{enumerate}

Now suppose that we have defined the sequence $\langle E(\tau'): \tau' < \tau  \rangle$. If $\langle E(\tau'): \tau' < \tau  \rangle \notin M^*$ we stop the construction and set

\begin{center}
$\forall \alpha \in \emph{A}, \bar{E}_{\alpha}= \langle \alpha, E(0), ..., E(\tau'), ...: \tau'< \tau \rangle$
\end{center}
and call $\bar{E}_{\alpha}$ \emph{an extender sequence of length $\tau$}  $(\len(
\Es_{\alpha})=\tau).$

If $\langle E(\tau'): \tau' < \tau  \rangle \in M^*$  then we define
an extender (with projections)

\begin{center}
$E(\tau)= \langle \langle E_{\langle \alpha, E(\tau'): \tau'< \tau \rangle}(\tau): \alpha \in \emph{A} \rangle, \langle \pi_{\langle \beta, E(\tau'): \tau'< \tau \rangle, \langle \alpha, E(\tau'): \tau'< \tau \rangle}: \beta, \alpha \in \emph{A}, \beta \geq_{j} \alpha \rangle \rangle$
\end{center}
on $V_{\kappa}$ by:

\begin{itemize}
\item $X \in E_{\langle \alpha, E(\tau'): \tau'< \tau \rangle}(\tau) \Leftrightarrow \langle \alpha, E(\tau'): \tau'< \tau \rangle \in j(X),$
\item for $\beta \geq_{j} \alpha$ in $\emph{A}, \pi_{\langle \beta, E(\tau'): \tau'< \tau \rangle, \langle \alpha, E(\tau'): \tau'< \tau \rangle}(\langle \nu, d \rangle)= \langle \pi_{\beta, \alpha}(\nu), d \rangle $
\end{itemize}

Note that $ E_{\langle \alpha, E(\tau'): \tau'< \tau \rangle}(\tau)$ concentrates on pairs of the form $\langle \nu, d \rangle$ where $\nu < \kappa$ and $d$ is an extender sequence. This makes the above definition well-defined.

We let the construction run until it stops due to the extender
sequence not being in $M^*$.

\begin{definition}
\begin{enumerate}
\item $\bar{\mu}$ is an extender sequence if there are $j: V^{*} \rightarrow M^{*}$ and $\bar{\nu}$ such that $\bar{\nu}$ is an extender sequence derived from $j$ as above (i.e $\bar{\nu}=\bar{E_{\alpha}}$ for some $\alpha$) and $\bar{\mu}=\bar{\nu}\upharpoonright \tau$ for some $\tau \leq \len(\bar{\nu}),$

\item $\kappa(\bar{\mu})$ is the ordinal of the beginning of the sequence (i.e $\kappa(\bar{E}_{\alpha})=\alpha$),

\item $\kappa^{0}(\bar{\mu})=(\kappa(\bar{\mu}))^{0}$ (i.e $\kappa^{0}(\bar{E}_{\alpha})= \kappa)$),

\item The sequence $ \langle \bar{\mu_{1}}, ..., \bar{\mu_{n}} \rangle$ of extender sequences is $^{0}-$increasing if $\kappa^{0}(\bar{\mu_1}) < ... < \kappa^{0}(\bar{\mu_n}),$

\item The extender sequence $\bar{\mu}$ is permitted to a $^{0}-$increasing sequence $ \langle \bar{\mu_{1}}, ..., \bar{\mu_{n}} \rangle$ of extender sequences if $\kappa^{0}(\bar{\mu_n})<\kappa^{0}(\bar{\mu}),$

\item Notation: We write $X \in \bar{E}_{\alpha}$ iff $\forall \xi < \len(\bar{E}_{\alpha}), X \in E_{\alpha}(\xi),$

\item $\bar{E}= \langle \bar{E}_{\alpha}: \alpha \in A  \rangle$ is an \emph{extender sequence system} if there is $j: V^{*} \rightarrow M^{*}$ such that each $\bar{E}_{\alpha}$ is derived from $j$ as above and $\forall \alpha, \beta \in A, \len(\bar{E}_{\alpha})= \len(\bar{E}_{\beta}).$ Call this common length, the length of $\bar{E}, \len(\bar{E}),$

\item For an extender sequence $\bar{\mu},$ we use $\bar{E}(\bar{\mu})$ for the extender sequence system containing $\bar{\mu}$ (i.e $\bar{E}(\bar{E}_{\alpha})= \bar{E}$),

\item $\dom(\bar{E})=A$,

\item $\bar{E}_{\beta}  \geq_{\bar{E}} \bar{E}_{\alpha} \Leftrightarrow \beta  \geq_{j} \alpha.$
\end{enumerate}
\end{definition}

\section{Finding generic filters}

Using $GCH$ in $V^*$ we construct  an extender sequence system $\bar{E}= \langle \bar{E}_{\alpha}: \alpha \in \dom\bar{E} \rangle$ where $\dom\bar{E}=[\kappa, \kappa^{+3})$ and $\len(\bar{E})=\kappa^{+}$ such that
the ultrapower $j_{\bar{E}}:V^{*}\rightarrow M_{\bar{E}}^{*}$ (defined
below) contains $V_{\kappa+3}^{*}.$ Suppose that $\bar{E}$ is derived from an elementary embedding $j: V^{*} \rightarrow M^{*}.$   Consider the following elementary embeddings $\forall \tau' < \tau < \len(\bar{E})$
\begin{align*} \label{E-system}
& j_\gt\func  \VS \to \MSt \simeq \Ult(\VS, E(\gt))=
\{j_\tau(f)(\bar E_\alpha \restricted \tau)\mid f\in V^*\},
\notag \\
&  k_\gt(j_\gt(f)(\Es_\ga \restricted \gt))=
        j(f)(\Es_\ga \restricted \gt),
\\
\notag & i_{\gt', \gt}(j_{\gt'}(f)(\Es_\ga \restricted \gt')) =
    j_\gt(f)(\Es_\ga \restricted \gt'),
\\
\notag & \ordered{\MSE,i_{\gt, \Es}} = \limdir \ordered {
        \ordof{\MSt} {\gt < \len(\Es)},
                \ordof{i_{\gt',\gt}} {\gt' \leq \gt < \len(\Es)}
        }.
\end{align*}
We demand that
        $\Es \restricted \gt \in \MSt$ for all $\tau<\len(\bar E)$.

Thus we get the following commutative diagram.

\begin{align*}
\begin{diagram}
\node{\VS}
        \arrow[3]{e,t}{j}
        \arrow{sse,l}{j_{\gt'}}
        \arrow[2]{se,l}{j_\gt}
        \arrow{seee,t,l}{j_{\Es}}
    \node{}
    \node{}
    \node{\MS}
\\
\node{}
    \node{}
    \node{}
        \node{\MSE}
        \arrow{n,r}{k_{\Es}}
\\
\node{}
    \node{M^*_{\gt'}}
        \arrow[2]{ne,t,3}{k_{\gt'}}
        \arrow{nee,t,2}{i_{\gt', \Es}}
        \arrow{e,b}{i_{\gt', \gt}}
    \node{\MSt = \Ult(\VS, E(\gt))}
        \arrow[1]{ne,b}{i_{\gt, \Es}}
        \arrow{nne,b,1}{k_{\gt}}
\end{diagram}
\end{align*}

Note that
\begin{itemize}
 \item the critical point of those elementary embeddings originating in $V^*$ is $\kappa,$  \item the critical point of those elementary embeddings originating in other models is $\kappa^{+4}$ as computed in that model.
\end{itemize}
Thus we get

\begin{align*}
& \crit i_{\gt',\gt} = \crit k_{\gt'} = \crit i_{\gt',\Es} =
         (\gk^{+4})_{M^*_{\gt'}},
\\
& \crit k_{\gt} = \crit(i_{\gt, \Es}) = (\gk^{+4})_{M^*_{\gt}},
\\
& \crit k_\Es = (\gk^{+4})_{\MSE}.
\end{align*}
Each of these models catches $V_{\gk^+3}^{\MS} = \VS_{\gk^+3}$ hence computes
$\gk^{+3}$ to be the same ordinal in all models. The larger $\gt$
is the more resemblance there is between $\MSt$ and $\MS$.
This can be verified by noting that
\begin{center}
 $ \kappa^{+4}_{M_{\tau^{'}}^{*}} < j_{\tau^{'}}(\kappa) < \kappa^{+4}_{M_{\tau}^{*}} < j_{\tau}(\kappa) < \kappa^{+4}_{M_{\bar{E}^{*}}} \leq \kappa^{+4}_{M^{*}} \leq \kappa^{+4}.$
\end{center}
We also factor through the normal ultrafilter to get the following commutative diagram

\begin{align*}
\begin{aligned}
\begin{diagram}
\node{\VS}
        \arrow{e,t}{j_\Es}
        \arrow{se,t}{j_\gt}
        \arrow{s,l}{i_U}
        \node{\MSE}
\\
\node{\NS \simeq \Ult(\VS, U)}
         \arrow{e,b}{i_{U, \gt}}
         \arrow{ne,b}{i_{U, \Es}}
        \node{\MSt}
         \arrow{n,b}{i_{\gt, \Es}}
\end{diagram}
\end{aligned}
\begin{aligned}
\qquad
\begin{split}
& U = E_\gk(0),
\\
& i_U \func  \VS \to \NS \simeq \Ult(\VS, U),
\\
& i_{U, \gt}(i_U(f)(\gk)) = j_\gt(f)(\gk),
\\
& i_{U, \Es}(i_U(f)(\gk)) = j_\Es(f)(\gk).
\end{split}
\end{aligned}
\end{align*}
$\NS$ catches $\VS$ only up to $\VS_{\gk+1}$ and we have
\begin{align*}
\gk^+ < \crit i_{U, \gt} = \crit i_{U, \Es}
    = \gk^{++}_{\NS} < i_U(\gk) < \gk^{++}.
\end{align*}

We now define the forcings for which we will need ``guiding
generics''.

\begin{definition}
Let
\begin{enumerate}
\item $\mathbb{R}_{U}^{\Col}=\Col(\kappa^{+6}, i_{U}(\kappa))_{N^{*}},$

\item  $\mathbb{R}_{U}^{\Add, 1}= \Add(\kappa^{+}, \kappa^{+4})_{N^{*}},$

\item  $\mathbb{R}_{U}^{\Add, 2}= \Add(\kappa^{++}, \kappa^{+5})_{N^{*}},$
\item  $\mathbb{R}_{U}^{\Add, 3}=  \Add(\kappa^{+3}, \kappa^{+6})_{N^{*}},$
\item  $\mathbb{R}_{U}^{\Add, 4}= (\Add(\kappa^{+4}, i_{U}(\kappa)^{+})  \times \Add(\kappa^{+5}, (i_{U}(\kappa)^{++})_{N^{*2}})  \times \Add(\kappa^{+6}, (i_{U}(\kappa)^{+3})_{N^{*2}}) )_{N^{*}},$ where $N^{*2}$ is the second iterate of $V^*$ by $U$,

\item $\mathbb{R}_{U}^{\Add}=\mathbb{R}_{U}^{\Add,1} \times \mathbb{R}_{U}^{\Add,2} \times \mathbb{R}_{U}^{\Add,3} \times \mathbb{R}_{U}^{\Add,4},$

\item $\mathbb{R}_{U} = \mathbb{R}_{U}^{\Add} \times \mathbb{R}_{U}^{\Col}.$
\end{enumerate}
\end{definition}
\begin{remark}
$(j_{\tau}(\kappa)^{++})_{M^*_{\tau}}=(j_{\tau}(\kappa)^{++})_{M^{*2}_{\tau}}$ and $(j_{\tau}(\kappa)^{+3})_{M^*_{\tau}}=(j_{\tau}(\kappa)^{+3})_{M^{*2}_{\tau}}$, where $M^{*2}_{\tau}$ is the second iterate of $V^*$ by $E(\tau).$ Similarly
$(j_{\bar{E}}(\kappa)^{++})_{M^*_{\bar{E}}}=(j_{\bar{E}}(\kappa)^{++})_{M^{*2}_{\bar{E}}}$ and $(j_{\bar{E}}(\kappa)^{+3})_{M^*_{\bar{E}}}=(j_{\bar{E}}(\kappa)^{+3})_{M^{*2}_{\bar{E}}}$, where $M^{*2}_{\bar{E}}$ is the second iterate of $V^*$ by $\bar{E}.$
\end{remark}

\begin{definition}
Let
\begin{enumerate}
\item $\mathbb{R}_{\tau}^{\Col}=\Col(\kappa^{+6}, j_{\tau}(\kappa))_{M_{\tau}^{*}},$

\item  $\mathbb{R}_{\tau}^{\Add, 1}= \Add(\kappa^{+}, \kappa^{+4})_{M_{\tau}^{*}},$

\item  $\mathbb{R}_{\tau}^{\Add, 2}= \Add(\kappa^{++}, \kappa^{+5})_{M_{\tau}^{*}},$
\item  $\mathbb{R}_{\tau}^{\Add, 3}= \Add(\kappa^{+3}, \kappa^{+6})_{M_{\tau}^{*}},$
\item  $\mathbb{R}_{\tau}^{\Add, 4}= (\Add(\kappa^{+4}, j_{\tau}(\kappa)^{+})  \times \Add(\kappa^{+5}, j_{\tau}(\kappa)^{++})  \times \Add(\kappa^{+6}, j_{\tau}(\kappa)^{+3}) )_{M_{\tau}^{*}},$

\item $\mathbb{R}_{\tau}^{\Add}=\mathbb{R}_{\tau}^{\Add,1} \times \mathbb{R}_{\tau}^{\Add,2} \times \mathbb{R}_{\tau}^{\Add,3} \times \mathbb{R}_{\tau}^{\Add,4},$
\item $\mathbb{R}_{\tau} = \mathbb{R}_{\tau}^{\Add} \times \mathbb{R}_{\tau}^{\Col}.$
\end{enumerate}
\end{definition}

\begin{definition}
Let
\begin{enumerate}
\item $\mathbb{R}_{\bar{E}}^{\Col}=\Col(\kappa^{+6}, j_{\bar{E}}(\kappa))_{M_{\bar{E}}^{*}},$
\item  $\mathbb{R}_{\bar{E}}^{\Add, 1}= \Add(\kappa^{+}, \kappa^{+4})_{M_{\bar{E}}^{*}},$

\item  $\mathbb{R}_{\bar{E}}^{\Add, 2}= \Add(\kappa^{++}, \kappa^{+5})_{M_{\bar{E}}^{*}},$
\item  $\mathbb{R}_{\bar{E}}^{\Add, 3}= \Add(\kappa^{+3}, \kappa^{+6})_{M_{\bar{E}}^{*}},$
\item  $\mathbb{R}_{\bar{E}}^{\Add, 4}= (\Add(\kappa^{+4}, j_{\bar{E}}(\kappa)^{+})  \times \Add(\kappa^{+5}, j_{\bar{E}}(\kappa)^{++})  \times \Add(\kappa^{+6}, j_{\bar{E}}(\kappa)^{+3}) )_{M_{\bar{E}}^{*}},$

\item $\mathbb{R}_{\bar{E}}^{\Add}=\mathbb{R}_{\bar{E}}^{\Add,1} \times \mathbb{R}_{\bar{E}}^{\Add,2} \times \mathbb{R}_{\bar{E}}^{\Add,3} \times \mathbb{R}_{\bar{E}}^{\Add,4},$
\item $\mathbb{R}_{\bar{E}} = \mathbb{R}_{\bar{E}}^{\Add} \times \mathbb{R}_{\bar{E}}^{\Col}.$
\end{enumerate}
\end{definition}
Also define the forcing notion $\mathbb{P}$ as follows
\begin{center}
 $\mathbb{P}=\mathbb{P}_{1} \times \mathbb{P}_{2} \times \mathbb{P}_{3} = \Add(\kappa^{+}, (\kappa^{+4})_{M^*_{\bar{E}}}) \times \Add(\kappa^{++}, (\kappa^{+5})_{M^*_{\bar{E}}}) \times \Add(\kappa^{+3}, (\kappa^{+6})_{M^*_{\bar{E}}}) $\footnote{ Hence $\mathbb{P}$ is forcing isomorphic to $\Add(\kappa^{+}, \kappa^{+3}) \times \Add(\kappa^{++}, \kappa^{+3}) \times \Add(\kappa^{+3}, \kappa^{+3})$.}
\end{center}
and let $G=G_{1} \times G_{2} \times G_{3}$ be $\mathbb{P}-$generic over $V^*$. It is clear that $V^{*}[G]$ is a cofinality-preserving generic extension of $V^{*}$ and that $GCH$ holds in $V^{*}[G]$ below and at $\kappa.$ The forcing $\mathbb P$ is our (weak) ``preparation
forcing'' (which preserves the $GCH$ below $\kappa$). We set $V=V^*[G]$.

\begin{remark}
$(a)$ It is also possible to work with  $\mathbb{P}= \Add(\kappa^{+}, \kappa^{+4}) \times \Add(\kappa^{++}, \kappa^{+5}) \times \Add(\kappa^{+3}, \kappa^{+6}).$

$(b)$ We also require that $G_2 \times G_3$ contains some special
element, that we will specify later (see the notes after Claim 3.13). The element will play the role of a master condition, and it will be used in the proof of Lemma 3.8$(d)$.
\end{remark}
\begin{lemma}
$(a)$ $G_{U}= \langle i_{U}^{''} G_{1} \rangle \times \langle i_{U}^{''} G_{2} \rangle \times \langle i_{U}^{''} G_{3} \rangle$ is $\mathbb{P}_{U}=i_{U}(\mathbb{P})-$generic over $N^{*}$,

$(b)$ $G_{\tau}= \langle j_{\tau}^{''} G_{1} \rangle \times \langle j_{\tau}^{''} G_{2} \rangle \times \langle j_{\tau}^{''} G_{3} \rangle$ is $\mathbb{P}_{\tau}=j_{\tau}(\mathbb{P})-$generic over $M_{\tau}^{*}$,

$(c)$ $G_{\bar{E}}= \langle \bigcup_{\tau < \len(\bar{E})} i_{\tau, \bar{E}}^{''}G_{\tau} \rangle $ is $\mathbb{P}_{\bar{E}}=j_{\bar{E}}(\mathbb{P})-$generic over $M_{\bar{E}}^{*}$,.
\end{lemma}
\begin{proof} $(a)$ Suppose $D \in N^{*}$ is dense open in $\mathbb{P}_{U}.$ Let $D=i_{U}(f)(\kappa)$ for some function $f \in V$ on $\kappa.$ Then
\begin{center}
 $D^{*}= \{ \alpha< \kappa: f(\alpha)$ is dense open in $\mathbb{P} \} \in U.$
\end{center}
Since $\mathbb{P}$ is $\kappa^{+}-$closed, $\bigcap_{\alpha \in D^{*}} f(\alpha)$ is dense open in $\mathbb{P}.$ Let $p \in G \cap \bigcap_{\alpha \in D^{*}} f(\alpha).$ Then $i_{U}(p) \in G_{U} \cap D.$

$(b)$ Suppose $D \in M_{\tau}^{*}$ is dense open in $\mathbb{P}_{\tau}.$ Let $D=j_{\tau}(f)(\bar{E_{\alpha}} \upharpoonright \tau)$ for some function $f \in V$ on $V_{\kappa}.$ Then
\begin{center}
 $D^{*}= \{ \bar{\nu} \in  V_{\kappa}: f(\bar{\nu})$ is dense open in $\mathbb{P} \} \in E_{\alpha}(\tau).$
\end{center}
Since $\mathbb{P}$ is $\kappa^{+}-$closed, $\bigcap_{\bar{\nu} \in D^{*}} f(\bar{\nu})$ is dense open in $\mathbb{P}.$ Let $p \in G \cap \bigcap_{\bar{\nu} \in D^{*}} f(\bar{\nu}).$ Then $j_{\tau}(p) \in G_{\tau} \cap D.$

$(c)$ Suppose $D \in M_{\bar{E}}^{*}$ is dense open in $\mathbb{P}_{\bar{E}}.$ Let $\tau < \len(\bar{E})$ and $D_{\tau} \in M_{\tau}^{*}$ be such that $D = i_{\tau, \bar{E}}(D_{\tau}).$ By elementarity $D_{\tau}$ is dense open in $\mathbb{P}_{\tau}.$ Let $p \in G_{\tau} \cap D_{\tau}.$ Then $i_{\tau, \bar{E}}(p) \in G_{\bar{E}} \cap D.$
\end{proof}
The following lemma is now trivial.
\begin{lemma}
The generic filters above are such that

$(a)$ $i_{U}^{''}[G] \subseteq G_{U},$

$(b)$ $j_{\tau}^{''}[G] \subseteq G_{\tau},$

$(c)$ $j_{\bar{E}}^{''}[G] \subseteq G_{\bar{E}},$

$(d)$ $i_{U, \tau^{'}}^{''}[G_{U}] \subseteq G_{\tau},$

$(e)$ $i_{\tau^{'}, \tau}^{''}[G_{\tau^{'}}] \subseteq G_{\tau},$

$(f)$ $i_{\tau, \bar{E}}^{''}[G_{\tau}] \subseteq G_{\bar{E}}.$
\end{lemma}

It then follows that we have the following lifting diagram.

\begin{align*}
\begin{diagram}
\node{V = \VS[G]}
        \arrow[2]{e,t}{j_\Es}
        \arrow{s,l}{i_{U}}
        \arrow{se,b}{j_{\gt'}}
        \arrow{see,b}{j_\gt}
    \node{}
        \node{\ME = \MSE[G_\Es]}
\\
    \node{N = \NS[G_U]}
         \arrow{e,b}{i_{U, \gt'}}
    \node{M_{\gt'} = M^*_{\gt'}[G_{\gt'}]}
         \arrow{ne,t,3}{i_{\gt', \Es}}
         \arrow{e,b}{i_{\gt', \gt}}
    \node{\Mt = \MSt[G_\gt]}
        \arrow[1]{n,b}{i_{\gt, \Es}}
\end{diagram}
\end{align*}

\begin{lemma}
In $V^{*}[G]$ there are $I_{U}, I_{\tau}$ and $I_{\bar{E}}$ such that

$(a)$ $I_{U}$ is $R_{U}-$generic over $N^{*}[G_{U}]$,

$(b)$ $I_{\tau}$ is $R_{\tau}-$generic over $M_{\tau}^{*}[G_{\tau}]$,

$(c)$ $I_{\bar{E}}$ is $R_{\bar{E}}-$generic over $M_{\bar{E}}^{*}[G_{\bar{E}}]$,

$(d)$ The generics are so that we have the following lifting diagram

\begin{align*}
\begin{diagram}
\node{}
    \node{}
        \node{\ME[I_\Es]}
\\
    \node{N[I_U]}
         \arrow{e,b}{i^*_{U, \gt'}}
    \node{M_{\gt'}[I_{\gt'}]}
         \arrow{ne,t,3}{i^*_{\gt', \Es}}
         \arrow{e,b}{i^*_{\gt', \gt}}
    \node{\Mt[I_\gt]}
        \arrow[1]{n,b}{i^*_{\gt, \Es}}
\end{diagram}
\end{align*}

\end{lemma}
\begin{proof}
We will prove the lemma in a sequence of claims.
\begin{claim} $I_{\bar{E}}^{\Add, 1} = G_{1} \cap \mathbb{R}_{\bar{E}}^{\Add, 1}$ is $\mathbb{R}_{\bar{E}}^{\Add, 1}-$generic over $M_{\bar{E}}^{*}.$
\end{claim}
\begin{proof}
Suppose $A$ is a maximal antichain of $\mathbb{R}_{\bar{E}}^{\Add, 1}$ in $M_{\bar{E}}^{*}.$ Let $X= \bigcup \{\dom(p): p \in A \}.$ As $|A| \leq \kappa^{+},$ we have $|X| \leq \kappa^{+},$ and of course $A$ is a maximal antichain  of $\Add(\kappa^{+}, X)_{M_{\bar{E}}^{*}}.$ For simplicity let us assume $|X|= \kappa^{+}.$ Now we have $\Add(\kappa^{+}, X)_{M_{\bar{E}}^{*}}=\Add(\kappa^{+}, X)$ and hence $A$ is a maximal antichain of $\Add(\kappa^{+}, X).$ It then follows that $A$ is a maximal antichain of $\Add(\kappa^{+}, \kappa^{+4}).$ Let $p \in G_{1} \cap A.$ Then $p \in  I_{\bar{E}}^{\Add, 1} \cap A.$
\end{proof}

\begin{claim} $I_{\bar{E}}^{\Add, 2} = G_{2}  \cap \mathbb{R}_{\bar{E}}^{\Add, 2}$ is $\mathbb{R}_{\bar{E}}^{\Add, 2}-$generic over $M_{\bar{E}}^{*}.$
\end{claim}
\begin{proof}
Suppose $A$ is a maximal antichain of $\mathbb{R}_{\bar{E}}^{\Add, 2}$ in $M_{\bar{E}}^{*}.$ Let $X= \bigcup \{\dom(p): p \in A \}.$ As $|A| \leq \kappa^{++},$ we have $|X| \leq \kappa^{++},$ and of course $A$ is a maximal antichain  of $\Add(\kappa^{++}, X)_{M_{\bar{E}}^{*}}.$ For simlpicity let us assume $|X|= \kappa^{++}.$ Now we have $\Add(\kappa^{++}, X)_{M_{\bar{E}}^{*}}=\Add(\kappa^{++}, X)$ and hence $A$ is a maximal antichain of $\Add(\kappa^{++}, X).$ It then follows that $A$ is a maximal antichain of $\Add(\kappa^{++}, \kappa^{+5}).$ Let $p \in G_{2} \cap A.$ Then $p \in  I_{\bar{E}}^{\Add, 2} \cap A.$
\end{proof}
\begin{claim} $I_{\bar{E}}^{\Add, 3} = G_{3}  \cap \mathbb{R}_{\bar{E}}^{\Add, 3}$ is $\mathbb{R}_{\bar{E}}^{\Add, 3}-$generic over $M_{\bar{E}}^{*}.$
\end{claim}
\begin{proof}
Suppose $A$ is a maximal antichain of $\mathbb{R}_{\bar{E}}^{\Add, 3}$ in $M_{\bar{E}}^{*}.$ Let $X= \bigcup \{\dom(p): p \in A \}.$ As $|A| \leq \kappa^{+3},$ we have $|X| \leq \kappa^{+3},$ and of course $A$ is a maximal antichain  of $\Add(\kappa^{+3}, X)_{M_{\bar{E}}^{*}}.$ For simlpicity let us assume $|X|= \kappa^{+3}.$ Now we have $\Add(\kappa^{+3}, X)_{M_{\bar{E}}^{*}}=\Add(\kappa^{+3}, X)$ and hence $A$ is a maximal antichain of $\Add(\kappa^{+3}, X).$ It then follows that $A$ is a maximal antichain of $\Add(\kappa^{+3}, \kappa^{+6}).$ Let $p \in G_{3} \cap A.$ Then $p \in  I_{\bar{E}}^{\Add, 3} \cap A.$
\end{proof}
It follows from the above claims that
\begin{claim}
$I_{\bar{E}}^{\Add, 1} \times I_{\bar{E}}^{\Add, 2} \times I_{\bar{E}}^{\Add, 3}$ is $\mathbb{R}_{\bar{E}}^{\Add, 1} \times \mathbb{R}_{\bar{E}}^{\Add, 2} \times \mathbb{R}_{\bar{E}}^{\Add, 3}-$generic over $M_{\bar{E}}^{*}.$
\end{claim}
\begin{claim}
$I_{U}^{\Add, 1} = \langle i_{U, \bar{E}}^{-1''}(I_{\bar{E}}^{\Add, 1}) \rangle$ is $\mathbb{R}_{U}^{\Add, 1}-$generic over $N^{*}.$
\end{claim}
\begin{proof}
Let $A$ be a maximal antichain of $\mathbb{R}_{U}^{\Add, 1}$ in $N^{*}.$ Then $i_{U, \bar{E}}(A)$ is a maximal antichain of $\mathbb{R}_{\bar{E}}^{\Add, 1}$ in $M_{\bar{E}}^{*}.$ Since $|A| \leq \kappa^{+},$ and $crit(i_{U, \bar{E}})=\kappa^{++}_{N^{*}} > \kappa^{+},$ we have  $i_{U, \bar{E}}(A)=i_{U, \bar{E}}^{''}(A).$ Then $I_{\bar{E}}^{\Add, 1} \cap i_{U, \bar{E}}^{''}(A) \neq \emptyset,$ which implies $I_{U}^{\Add, 1} \cap A \neq \emptyset.$
\end{proof}
Now consider the forcing notion $\mathbb{R}_{U}^{\Add, 2} \times \mathbb{R}_{U}^{\Add, 3} \times \mathbb{R}_{U}^{\Add, 4} \times \mathbb{R}_{U}^{\Col}$. Working in $M_{\bar{E}}^{*},$ this forcing notion is $\kappa^{+}-$closed and there are only $\kappa^{+}-$many maximal antichains of it which are in $N^{*}.$ Thus we can define a descending sequence $\langle \langle p_{\langle \alpha, \Add, 2 \rangle}, p_{\langle \alpha, \Add, 3 \rangle}, p_{\langle \alpha, \Add, 4 \rangle}, p_{\langle \alpha, \Col \rangle} \rangle  : \alpha < \kappa^{+}  \rangle$ of conditions such that $I_{U}^{\Add, 2} \times I_{U}^{\Add, 3} \times I_{U}^{\Add, 4} \times I_{U}^{\Col} = \{p \in \mathbb{R}_{U}^{\Add, 2} \times \mathbb{R}_{U}^{\Add, 3} \times \mathbb{R}_{U}^{\Add, 4} \times \mathbb{R}_{U}^{\Col}: \exists \alpha< \kappa^{+}, \langle p_{\langle \alpha, \Add, 2 \rangle}, p_{\langle \alpha, \Add, 3 \rangle}, p_{\langle \alpha, \Add, 4 \rangle}, p_{\langle \alpha, \Col \rangle} \rangle \leq p  \}$ is $\mathbb{R}_{U}^{\Add, 2} \times \mathbb{R}_{U}^{\Add, 3} \times \mathbb{R}_{U}^{\Add, 4} \times \mathbb{R}_{U}^{\Col}$-generic over $N^{*}.$ Also note that this generic filter is in $M_{\bar{E}}^{*}.$

We may also note that $\{ i_{U, \bar{E}}(p_{\langle \alpha, \Add, 2 \rangle}, p_{\langle \alpha, \Add, 3 \rangle}): \alpha < \kappa^{+} \} \subseteq \mathbb{R}_{\bar{E}}^{\Add, 2} \times \mathbb{R}_{\bar{E}}^{\Add, 3},$ and since this forcing is $\kappa^{++}-$closed, there is $\langle p_{\langle \Add, 2 \rangle}, p_{\langle \Add, 3 \rangle} \rangle \in \mathbb{R}_{\bar{E}}^{\Add, 2} \times \mathbb{R}_{\bar{E}}^{\Add, 3}$ such that $\forall \alpha < \kappa^{+}, \langle p_{\langle \Add, 2 \rangle}, p_{\langle \Add, 3 \rangle} \rangle \leq
i_{U, \bar{E}}(p_{\langle \alpha, \Add, 2 \rangle}, p_{\langle \alpha, \Add, 3 \rangle}).$ We may suppose that $\langle p_{\langle \Add, 2\rangle}, p_{\langle \Add, 3 \rangle} \rangle \in G_{2} \times G_{3}$ (see Remark 3.5$(b)$).

Let $I_{U}= I_{U}^{\Add, 1} \times I_{U}^{\Add, 2} \times I_{U}^{\Add, 3} \times I_{U}^{\Add, 4} \times I_{U}^{\Col}.$ It follows from the above results that $I_{U}$ is $\mathbb{R}_{U}-$generic over $N^{*}.$

\begin{claim}
$(a)$ $I_{\tau}^{\Add, 1} = \langle i_{\tau, \bar{E}}^{-1''}(I_{\bar{E}}^{\Add, 1}) \rangle$ is $\mathbb{R}_{\tau}^{\Add, 1}-$generic over $M_{\tau}^{*},$

$(b)$ $I_{\tau}^{\Add, 2} = \langle i_{\tau, \bar{E}}^{-1''}(I_{\bar{E}}^{\Add, 2}) \rangle$ is $\mathbb{R}_{\tau}^{\Add, 2}-$generic over $M_{\tau}^{*},$

$(c)$ $I_{\tau}^{\Add, 3} = \langle i_{\tau, \bar{E}}^{-1''}(I_{\bar{E}}^{\Add, 3}) \rangle$ is $\mathbb{R}_{\tau}^{\Add, 3}-$generic over $M_{\tau}^{*}.$

\end{claim}
\begin{proof}
$(a)$ Let $A$ be a maximal antichain of $\mathbb{R}_{\tau}^{\Add, 1}$ in $M_{\tau}^{*}.$ Then $i_{\tau, \bar{E}}(A)$ is a maximal antichain of $\mathbb{R}_{\bar{E}}^{\Add, 1}$ in $M_{\bar{E}}^{*}.$ Since $|A| \leq \kappa^{+},$ and $crit(i_{\tau, \bar{E}})=\kappa^{+4}_{M_{\tau}^{*}} > \kappa^{+},$ we have  $i_{\tau, \bar{E}}(A)=i_{\tau, \bar{E}}^{''}(A).$ Then $I_{\bar{E}}^{\Add, 1} \cap i_{\tau, \bar{E}}^{''}(A) \neq \emptyset,$ which implies $I_{\tau}^{\Add, 1} \cap A \neq \emptyset.$

$(b)$ Let $A$ be a maximal antichain of $\mathbb{R}_{\tau}^{\Add, 2}$ in $M_{\tau}^{*}.$ Then $i_{\tau, \bar{E}}(A)$ is a maximal antichain of $\mathbb{R}_{\bar{E}}^{\Add, 2}$ in $M_{\bar{E}}^{*}.$ Since $|A| \leq \kappa^{++},$ and $crit(i_{\tau, \bar{E}})=\kappa^{+4}_{M_{\tau}^{*}} > \kappa^{++},$ we have  $i_{\tau, \bar{E}}(A)=i_{\tau, \bar{E}}^{''}(A).$ Then $I_{\bar{E}}^{\Add, 2} \cap i_{\tau, \bar{E}}^{''}(A) \neq \emptyset,$ which implies $I_{\tau}^{\Add, 2} \cap A \neq \emptyset.$

$(c)$ Let $A$ be a maximal antichain of $\mathbb{R}_{\tau}^{\Add, 3}$ in $M_{\tau}^{*}.$ Then $i_{\tau, \bar{E}}(A)$ is a maximal antichain of $\mathbb{R}_{\bar{E}}^{\Add, 3}$ in $M_{\bar{E}}^{*}.$ Since $|A| \leq \kappa^{+3},$ and $crit(i_{\tau, \bar{E}})=\kappa^{+4}_{M_{\tau}^{*}} > \kappa^{+3},$ we have  $i_{\tau, \bar{E}}(A)=i_{\tau, \bar{E}}^{''}(A).$ Then $I_{\bar{E}}^{\Add, 3} \cap i_{\tau, \bar{E}}^{''}(A) \neq \emptyset,$ which implies $I_{\tau}^{\Add, 1} \cap A \neq \emptyset.$
\end{proof}

As $U \in M_{\bar{E}}^{*}$ and $\forall \gt < \len(\Es)$ $E(\gt) \in M_{\bar{E}}^{*}$
we have the following diagram
\begin{align*}
\begin{aligned}
\begin{diagram}
\node{M_{\bar{E}}^{*}}
        \arrow{s,l}{i_U^\Es}
        \arrow{se,t}{j_\gt^\Es}
\\
\node{N^{* \bar{E}}}
         \arrow{e,b}{i_{U, \gt}^\Es}
        \node{M_{\tau}^{*\bar{E}}}
\end{diagram}
\end{aligned}
\begin{aligned}
\begin{split}
& U = E_\gk(0),
\\
& i^\Es_U \func  M_{\bar{E}}^{*} \to N^{* \bar{E}} \simeq \Ult(M_{\bar{E}}^{*}, U),
\\
& j^\Es_\gt \func  M_{\bar{E}}^{*} \to M_{\tau}^{*\bar{E}} \simeq \Ult(M_{\bar{E}}^{*}, E(\gt)),
\\
& i^\Es_{U, \gt}(i^\Es_U(f)(\gk)) = j^\Es_\gt(f)(\gk).
\end{split}
\end{aligned}
\end{align*}
Recall that we have $I_{U}^{\Add, 4} \times I_{U}^{\Col} \in M_{\bar{E}}^{*}$ which is $\mathbb{R}_{U}^{\Add, 4} \times \mathbb{R}_{U}^{\Col}$-generic over $N^{* \bar{E}}$.

\begin{claim}
There is $I_{\tau}^{\Add, 4} \times I_{\tau}^{\Col} \in M_{\bar{E}}^{*}$ which is $\mathbb{R}_{\tau}^{\Add, 4} \times \mathbb{R}_{\tau}^{\Col}-$generic over $M_{\tau}^{*}.$
\end{claim}
\begin{proof}
We follow the idea from [3]. For this set
\begin{enumerate}
\item $\mathbb{R}_{\tau}^{\bar{E}, \Col}=\Col(\kappa^{+6}, j_{\tau}^{\bar{E}}(\kappa))_{M_{\tau}^{* \bar{E}}},$

\item  $\mathbb{R}_{\tau}^{\bar{E}, \Add, 4}= (\Add(\kappa^{+4}, j_{\tau}^{\bar{E}}(\kappa)^{+})  \times \Add(\kappa^{+5}, j_{\tau}^{\bar{E}}(\kappa)^{++})  \times \Add(\kappa^{+6}, j_{\tau}^{\bar{E}}(\kappa)^{+3}) )_{M_{\tau}^{* \bar{E}}},$

\end{enumerate}
$\mathbb{R}_{\tau}^{\bar{E}, \Add, 4} \times \mathbb{R}_{\tau}^{\bar{E}, \Col}$ and $\mathbb{R}_{\tau}^{ \Add, 4} \times \mathbb{R}_{\tau}^{ \Col}$  are coded in $V^{\MSt}_{j_\gt(\gk)+3}$,
    $V^{\MStE}_{j^\Es_\gt(\gk)+3}$ respectively.
$V^{\MSt}_{j_\gt(\gk)+3}$, $V^{\MStE}_{j^\Es_\gt(\gk)+3}$
are determined by $V^{V^*}_{\gk+3}$, $V^{\MSE}_{\gk+3}$ (and $E(\gt)$,
of course).
As $E(\gt) \in \MSE$ and $V^{V^*}_{\gk+3} = V^\MSE_{\gk+3}$ we get that
$\mathbb{R}_{\tau}^{\bar{E}, \Add, 4} \times \mathbb{R}_{\tau}^{\bar{E}, \Col}=\mathbb{R}_{\tau}^{ \Add, 4} \times \mathbb{R}_{\tau}^{ \Col}.$

By the same reasoning, each antichain of $\mathbb{R}_{\tau}^{ \Add, 4} \times \mathbb{R}_{\tau}^{ \Col}$ appearing in
$M_{\tau}^{*}$
is also an anti-chain of $\mathbb{R}_{\tau}^{\bar{E}, \Add, 4} \times \mathbb{R}_{\tau}^{\bar{E}, \Col}$ appearing in $M_{\tau}^{*\bar{E}}$.
Hence, if $I_{\tau}^{\Add, 4} \times I_{\tau}^{\Col} \in M_{\bar{E}}^{*}$ is
$\mathbb{R}_{\tau}^{\bar{E}, \Add, 4} \times \mathbb{R}_{\tau}^{\bar{E}, \Col}-$generic filter over $M_{\tau}^{*\bar{E}}$
then it is also $\mathbb{R}_{\tau}^{\Add, 4} \times \mathbb{R}_{\tau}^{\Col}-$generic over $M_{\tau}^{*}.$

Let $I_{\tau}^{\Add, 4} \times I_{\tau}^{\Col} = \langle i_{U, \tau}^{\bar{E}''}(I_{U}^{\Add, 4} \times I_{U}^{\Col} )  \rangle.$ We show that it is as required. So let $D \in M_{\tau}^{*\bar{E}}$ be dense open in $\mathbb{R}_{\tau}^{\bar{E}, \Add, 4} \times \mathbb{R}_{\tau}^{\bar{E}, \Col}.$ Then $D=j_{\tau}^{\bar{E}}(f)(\bar{E}_{\alpha}\upharpoonright \tau)$ for some function $f \in M_{\bar{E}}^{*}$ on $V_{\kappa}.$ It then follows that in $M_{\bar{E}}^{*}$
\begin{center}
 $D^{*}=\{\bar{\nu} \in V_{\kappa}: f(\bar{\nu})$ is dense open in $\mathbb{R}_{\bar{E}}^{\Add, 4} \times \mathbb{R}_{\bar{E}}^{\Col} \} \in E_{\alpha}(\tau).$
\end{center}
It is easily seen that
\begin{center}
 $B = \{\mu: |\{\bar{\nu} \in D^{*}: \kappa^{0}(\bar{\nu})=\mu \}| \leq \mu^{+3} \} \in E_{\kappa}(0).$
\end{center}
Thus for each $\mu \in B$ we can find $f^{*}(\mu)$ such that for all $\bar{\nu} \in D^{*}$ with $\kappa^{0}(\bar{\nu})=\mu$ we have $f^{*}(\mu) \subseteq f(\bar{\nu})$ is dense open (in $M_{\bar{E}}^{*}$). Hence
\begin{center}
 $N^{*\bar{E}} \models `` i_{U}^{\bar{E}}(f^{*})(\kappa)$ is dense open ''.
\end{center}
and
\begin{center}
 $M_{\tau}^{*\bar{E}} \models `` j_{\tau}^{\bar{E}}(f^{*})(\kappa) \subseteq j_{\tau}^{\bar{E}}(f)(\bar{E}_{\alpha}\upharpoonright\tau)$''.
\end{center}
So there is $g \in M_{\bar{E}}^{*}$ such that $i_{U}^{\bar{E}}(g)(\kappa) \in (I_{U}^{\Add, 4} \times I_{U}^{\Col}) \cap i_{U}^{\bar{E}}(f^{*})(\kappa). $ It then follows that $j_{\tau}^{\bar{E}}(g)(\kappa) \in i_{U, \tau}^{\bar{E}''}(I_{U}^{\Add, 4} \times I_{U}^{\Col}) \cap j_{\tau}^{\bar{E}}(f^{*})(\kappa). $
This means that $(I_{\tau}^{\Add, 4} \times I_{\tau}^{\Col})    \cap j_{\tau}^{\bar{E}}(f)(\bar{E}_{\alpha}\upharpoonright \tau) \neq \emptyset.$
The result follows.

\end{proof}

Now let $I_{\tau}= I_{\tau}^{\Add, 1} \times I_{\tau}^{\Add, 2} \times I_{\tau}^{\Add, 3} \times I_{\tau}^{\Add, 4} \times I_{\tau}^{\Col}.$ It follows that $I_{\tau}$ is $\mathbb{R}_{\tau}-$generic over $M_{\tau}^{*}.$

\begin{claim}
$I_{\bar{E}}^{\Add, 4} \times I_{\bar{E}}^{\Col} = \langle \bigcup_{\tau < \len(\bar{E})}i_{\tau, \bar{E}}^{''} (I_{\tau}^{\Add, 4} \times I_{\tau}^{\Col}) \rangle$ is $\mathbb{R}_{\bar{E}}^{\Add, 4} \times \mathbb{R}_{\bar{E}}^{\Col}-$generic over $M_{\bar{E}}^{*}.$
\end{claim}
\begin{proof}
 Let $D$ be a dense open subset of $\mathbb{R}_{\bar{E}}^{\Add, 4} \times \mathbb{R}_{\bar{E}}^{\Col}$ in $M_{\bar{E}}^{*}.$ Let $\tau < l(\bar{E})$ and $D_{\tau} \in M_{\tau}^{*}$ be such that $D = i_{\tau, \bar{E}}(D_{\tau}).$ By elementarity $D_{\tau}$ is dense open in $\mathbb{R}_{\tau}^{\Add, 4} \times \mathbb{R}_{\tau}^{\Col}.$ Let $\langle p, q \rangle \in (I_{\tau}^{\Add, 4} \times I_{\tau}^{\Col}) \cap D_{\tau}.$ Then $\langle i_{\tau, \bar{E}}(p), i_{\tau, \bar{E}}(q) \rangle \in (I_{\bar{E}}^{\Add, 4} \times I_{\bar{E}}^{\Col}) \cap D.$
\end{proof}

Let $I_{\bar{E}}= I_{\bar{E}}^{\Add, 1} \times I_{\bar{E}}^{\Add, 2} \times I_{\bar{E}}^{\Add, 3} \times I_{\bar{E}}^{\Add, 4} \times I_{\bar{E}}^{\Col}.$ It follows that $I_{\bar{E}}$ is $\mathbb{R}_{\bar{E}}-$generic over $M_{\bar{E}}^{*}.$

To summarize, so far we have shown the following
\begin{itemize}
\item $I_{U}$ is $\mathbb{R}_{U}-$generic over $N^{*}$, \item $I_{\tau}$ is $\mathbb{R}_{\tau}-$generic over $M_{\tau}^{*}$, \item $I_{\bar{E}}$ is $\mathbb{R}_{\bar{E}}-$generic over $M_{\bar{E}}^{*}.$
\end{itemize}
Before continuing we recall Easton's lemma.
\begin{lemma}
(Easton's Lemma). Let $\lambda$ be regular uncountable, and suppose that $\mathbb{P}$ satisfies the $\lambda-c.c.$ and $\mathbb{Q}$ is $\lambda-$closed. Then

$(a)$ $\vdash_{\mathbb{P} \times \mathbb{Q}} `` \lambda$ is a regular uncountable cardinal'',

$(b)$  $\vdash_{\mathbb{Q}} `` \mathbb{P}$ satisfies the $\lambda-c.c. $'',

$(c)$ $\vdash_{\mathbb{P}} `` \mathbb{Q}$ is $\lambda-$distributive''.
\end{lemma}

\begin{claim}
$I_{U}$ is $\mathbb{R}_{U}-$generic over $N^{*}[G_{U}].$
\end{claim}
\begin{proof}
First note that in $N^{*}$ the forcing notions $\mathbb{R}_{U}$ and $\mathbb{P}_{U}$ are $i_{U}(\kappa)^{+}-c.c$ and $i_{U}(\kappa)^{+}-$closed respectively. Now let $A$ be a maximal antichain of $\mathbb{R}_{U}$ in $N^{*}[G].$ By Easton's Lemma $|A| \leq i_{U}(\kappa),$ hence again by Easton's Lemma $A \in N^{*}.$ It follows that $I_{U} \cap A \neq \emptyset,$ as $I_{U}$ is $\mathbb{R}_{U}-$generic over $N^{*}.$ The result follows.
\end{proof}
By similar arguments
\begin{claim}
$I_{\tau}$ is $\mathbb{R}_{\tau}-$generic over $M_{\tau}^{*}[G_{\tau}].$
\end{claim}
\begin{claim}
$I_{\bar{E}}$ is $\mathbb{R}_{\bar{E}}-$generic over $M_{\bar{E}}^{*}[G_{\bar{E}}].$
\end{claim}
It remains to prove part $(d)$ of lemma 3.8.
Before going into details let's recall a simple observation.
\begin{claim}
$(a)$ $V^{\NSE}_{i_{U}^\Es(\gk)+3} = V^\NS_{i_{U}(\gk)+3},$

$(b)$  $i^\Es_{U, \gt} \restricted
                V^\NSE_{i_{U}^\Es(\gk)+3} =
                    i_{U, \gt} \restricted
                V^\NS_{i_{U}(\gk)+3}$,

$(c)$  $i^\Es_{\gt', \gt} \restricted
                V^\NSE_{i_{U}^\Es(\gk)+3} =
                    i_{\gt', \gt} \restricted   V^\NS_{i_{U}(\gk)+3}$.
\end{claim}

\begin{claim}
$i_{U, \tau^{'}}^{''}(I_{U}) \subseteq I_{\tau^{'}}.$
\end{claim}
\begin{proof}
We have
\begin{enumerate}
\item $i_{U, \tau^{'}}^{''}(I_{U}^{\Add, 1}) \subseteq I_{\tau^{'}}^{\Add, 1}:$ This is because $i_{U, \tau^{'}}^{''}(I_{U}^{\Add, 1})=i_{U, \tau^{'}}^{''} (\langle i_{U, \bar{E}}^{-1''}(I_{\bar{E}}^{\Add, 1}) \rangle) \subseteq  \langle i_{\tau^{'}, \bar{E}}^{-1''}(I_{\bar{E}}^{\Add, 1}) \rangle =I_{\tau^{'}}^{\Add, 1},$
\item  $i_{U, \tau^{'}}^{''}(I_{U}^{\Add, 2}) \subseteq I_{\tau^{'}}^{\Add, 2}:$ It suffices to show that $\forall \alpha < \kappa^{+}, i_{U, \tau^{'}}(p_{\langle \alpha, \Add, 2  \rangle}) \in I_{\tau^{'}}^{\Add, 2}.$ But we have $p_{\langle \Add, 2 \rangle} \in I_{\bar{E}}^{\Add, 2}$ and $\forall \alpha < \kappa^{+}, p_{\langle \Add, 2 \rangle} \leq  i_{U, \bar{E}}(p_{\langle \alpha, \Add, 2 \rangle}).$ It then follows that $\forall \alpha < \kappa^{+}, i_{U, \tau^{'}}(p_{\langle \alpha, \Add, 2 \rangle})=i_{\tau^{'}, \bar{E}}^{-1}(i_{U, \bar{E}}(p_{\langle \alpha, \Add, 2 \rangle})) \geq i_{\tau^{'}, \bar{E}}^{-1}(p_{\langle \Add, 2 \rangle}).$ But now note that by our definition $i_{\tau^{'}, \bar{E}}^{-1}(p_{\langle \Add, 2 \rangle}) \in I_{\tau^{'}}^{\Add, 2}.$ It then follows that $i_{U, \tau^{'}}(p_{\langle \alpha, \Add, 2  \rangle}) \in I_{\tau^{'}}^{\Add, 2},$

\item  $i_{U, \tau^{'}}^{''}(I_{U}^{\Add, 3}) \subseteq I_{\tau^{'}}^{\Add, 3}:$ By the same argument as in $(2)$ using the fact that  $p_{\langle \Add, 3 \rangle} \in I_{\bar{E}}^{\Add, 3} ,$
\item  $i_{U, \tau^{'}}^{''}(I_{U}^{\Add, 4} \times I_{U}^{\Col}) \subseteq I_{\tau^{'}}^{\Add, 4} \times I_{\tau^{'}}^{\Col}:$ Trivial by the definition of $I_{\tau^{'}}^{\Add, 4} \times I_{\tau^{'}}^{\Col}$ and the previous Claim.
\end{enumerate}
The result follows.
\end{proof}

\begin{claim}
$i_{\tau^{'}, \tau}^{''}(I_{\tau^{'}}) \subseteq I_{\tau}.$
\end{claim}
\begin{proof}
We have
\begin{enumerate}
 \item $i_{\tau^{'}, \tau}^{''}(I_{\tau^{'}}^{\Add, 1}) \subseteq I_{\tau}^{\Add, 1}:$ Because $i_{\tau^{'}, \tau}^{''}(I_{\tau^{'}}^{\Add, 1})=i_{\tau^{'}, \tau}^{''} (\langle i_{\tau^{'}, \Es}^{-1''}(I_{\bar{E}}^{\Add, 1}) \rangle) \subseteq  \langle i_{\tau, \bar{E}}^{-1''}(I_{\bar{E}}^{\Add, 1}) \rangle =I_{\tau}^{\Add, 1},$
  \item $i_{\tau^{'}, \tau}^{''}(I_{\tau^{'}}^{\Add, 2}) \subseteq I_{\tau}^{\Add, 2}:$ By the same argument as in $(1),$
 \item $i_{\tau^{'}, \tau}^{''}(I_{\tau^{'}}^{\Add, 3}) \subseteq I_{\tau}^{\Add, 3}:$ By the same argument as in $(1),$
 \item $i_{\tau^{'}, \tau}^{''}(I_{\tau^{'}}^{\Add, 4} \times I_{\tau^{'}}^{\Col}) \subseteq I_{\tau}^{\Add, 4} \times I_{\tau}^{\Col}:$ Because $i_{\tau^{'}, \tau}^{''}(I_{\tau^{'}}^{\Add, 4} \times I_{\tau^{'}}^{\Col}) =i_{\tau^{'}, \tau}^{''}( \langle i_{U, \tau^{'}}^{\bar{E}''}(I_{U}^{\Add, 4} \times I_{U}^{\Col})  \rangle) \subseteq \langle i_{U, \tau}^{\bar{E}''}(I_{U}^{\Add, 4} \times I_{U}^{\Col})  \rangle =I_{\tau}^{\Add, 4} \times I_{\tau}^{\Col}.$

\end{enumerate}
The result follows.

\end{proof}
\begin{claim}
$i_{\tau, \bar{E}}^{''}(I_{\tau}) \subseteq I_{\bar{E}}.$
\end{claim}
\begin{proof}
We have
\begin{enumerate}
 \item $i_{\tau, \bar{E}}^{''}(I_{\tau}^{\Add, 1}) \subseteq I_{\bar{E}}^{\Add, 1}:$ Trivial as  $i_{\tau, \bar{E}}^{''}(I_{\tau}^{\Add, 1}) =i_{\tau, \bar{E}}^{''}(\langle i_{\tau, \bar{E}}^{-1''}(I_{\tau}^{\Add, 1}) \rangle) \subseteq I_{\bar{E}}^{\Add, 1},$
\item $i_{\tau, \bar{E}}^{''}(I_{\tau}^{\Add, 2}) \subseteq I_{\bar{E}}^{\Add, 2}:$ As in $(1),$
\item $i_{\tau, \bar{E}}^{''}(I_{\tau}^{\Add, 3}) \subseteq I_{\bar{E}}^{\Add, 3}:$ As in $(1),$
\item $i_{\tau, \bar{E}}^{''}(I_{\tau}^{\Add, 4} \times I_{\tau}^{\Col}) \subseteq I_{\bar{E}}^{\Add, 4} \times I_{\bar{E}}^{\Col}:$ Trivial by the definition of $I_{\bar{E}}^{\Add, 4} \times I_{\bar{E}}^{\Col}.$
\end{enumerate}
The result follows.

\end{proof}
This completes the proof of lemma 3.8.
\end{proof}

We iterate $j_\Es$ $\omega$-many times and consider the following diagram
\begin{align*}
\begin{diagram}
\node{V}
        \arrow[2]{e,t}{j_\Es = j^{0,1}_\Es}
        \arrow{se,t,1}{j_{\gt_1}}
        \arrow{s,l}{i_{U}}
    \node{}
        \node{\ME}
        \arrow[2]{e,t}{j^{1,2}_\Es}
        \arrow{se,t,1}{j^2_{\gt_2}}
        \arrow{s,l}{i^2_{U}}
    \node{}
        \node{M_\Es^2}
        \arrow[2]{e,t}{j^{2,3}_\Es}
        \arrow{se,t,1}{j^3_{\gt_3}}
        \arrow{s,l}{i^3_{U}}
    \node{}
        \node{M_\Es^3}
        \arrow[1]{e,..}
\\
\node{N}
        \arrow[1]{e,b}{i_{U, \gt_1}}
        \arrow[1]{nee,t,3}{i_{U, \Es}}
    \node{M_{\gt_1}}
        \arrow[1]{ne,b,1}{i_{\gt_1, \Es}}
    \node{N^2}
        \arrow[1]{e,b}{i^2_{U, \gt_2}}
        \arrow[1]{nee,t,3}{i^2_{U, \Es}}
    \node{M^{2}_{\gt_2}}
        \arrow[1]{ne,b,1}{i^2_{\gt_2, \Es}}
    \node{N^3}
        \arrow[1]{e,b}{i^3_{U, \gt_3}}
        \arrow[1]{nee,t,3}{i^3_{U, \Es}}
    \node{M^{3}_{\gt_3}}
        \arrow[1]{ne,b,1}{i^3_{\gt_3, \Es}}
\end{diagram}
\end{align*}
where
\begin{align*}
& j^0_\Es = \id,
\\
& j^{n}_\Es = j^{0, n}_\Es,
\\
& j^{m, n}_\Es = j^{n-1, n}_\Es \circ \dotsb \circ j^{m+1, m+2}_\Es \circ j^{m ,m+1}_\Es.
\end{align*}

Let $R(-,-)=R^{\Add}(-,-) \times R^{\Col}(-,-)$ be a function such that
\begin{center}
$i_{U}^{2}(R^{\Add})(\kappa, i_{U}(\kappa))=\mathbb{R}_{U}^{\Add},$

$i_{U}^{2}(R^{\Col})(\kappa, i_{U}(\kappa))=\mathbb{R}_{U}^{\Col},$
\end{center}
where $i_{U}^{2}$ is the second iterate of $i_{U}.$ Then we will have
\begin{center}
$i_{U}^{2}(R)(\kappa, i_{U}(\kappa))=\mathbb{R}_{U}.$
\end{center}
The following is trivial.
\begin{lemma}
$(a)$ $j_{\bar{E}}^{2}(R^{\Add})(\kappa, j_{\bar{E}}(\kappa))=\mathbb{R}_{\bar{E}}^{\Add},$

$(b)$ $j_{\bar{E}}^{2}(R^{\Col})(\kappa, j_{\bar{E}}(\kappa))=\mathbb{R}_{\bar{E}}^{\Col},$

$(c)$ $j_{\bar{E}}^{2}(R)(\kappa, j_{\bar{E}}(\kappa))=\mathbb{R}_{\bar{E}}.$
\end{lemma}

{\bf Cardinal structure and the power function in} {\bf $N^{*}[I_{U}]$.} The following lemma gives us everything that we need about the model $N^{*}[I_{U}]$.
\begin{lemma}
$(a)$ In $N^{*}[I_{U}]$ there are no cardinals in $[\kappa^{+7}, i_{U}(\kappa)]$ and all other $N^{*}-$cardinals are preserved,

$(b)$ The power function differs from the power function of $N^{*}$ at the following points: $2^{\kappa^{+}}=\kappa^{+4}, 2^{\kappa^{++}}=\kappa^{+5}, 2^{\kappa^{+3}}=\kappa^{+6}, 2^{\kappa^{+4}}=i_{U}(\kappa)^{+}, 2^{\kappa^{+5}}=i_{U}(\kappa)^{++}, 2^{\kappa^{+6}}=i_{U}(\kappa)^{+3}.$
\end{lemma}

{\bf Cardinal} {\bf structure} {\bf in} {\bf $M_{\tau}^{*}[I_{\tau}]$} and {\bf $M_{\bar{E}}^{*}[I_{\bar{E}}]$.} The following lifting says everything which we can possibly say.

\begin{align*}
\begin{diagram}
\node{}
    \node{}
        \node{M_{\bar{E}}^{*}[I_\Es]}
\\
    \node{N^{*}[I_U]}
         \arrow{e,b}{i^*_{U, \gt'}}
    \node{M_{\gt'}^{*}[I_{\gt'}]}
         \arrow{ne,t,3}{i^*_{\gt', \Es}}
         \arrow{e,b}{i^*_{\gt', \gt}}
    \node{M_{\tau}^{*}[I_\gt]}
        \arrow[1]{n,b}{i^*_{\gt, \Es}}
\end{diagram}
\end{align*}

\bigskip

The forcing notion $\mathbb{P}_{\bar{E}}$, due to Merimovich, which we define later, adds a club to $\gk$.
For each $\gn_1, \gn_2$ successive points in the club
the cardinal structure and power function in the range $[\gn_1^+, \gn_2^{+3}]$
of the generic extension
is analogous to the cardinal structure and power function in the range
$[\gk^+, j_\Es(\gk)^{+3}]$ of $M_{\bar{E}}^{*}[I_\Es]$.

{\bf Cardinal} {\bf structure} {\bf in} {\bf $N^{*}[I_{U}^{\Col}]$.} The following lemma gives us everything that we need about the model $N^{*}[I_{U}^{\Col}]$.
\begin{lemma}
$(a)$ In $N^{*}[I_{U}^{\Col}]$ there are no cardinals in $[\kappa^{+7}, i_{U}(\kappa)]$ and all other $N^{*}-$cardinals are preserved,

$(b)$ $GCH$ holds in $N^{*}[I_{U}^{\Col}].$
\end{lemma}

{\bf Cardinal} {\bf structure} {\bf in} {\bf $M_{\tau}^{*}[I_{\tau}^{\Col}]$} and {\bf $M_{\bar{E}}^{*}[I_{\bar{E}}^{\Col}]$.} The following lifting says everything which we can possibly say.

\begin{align*}
\begin{diagram}
\node{}
    \node{}
        \node{M_{\bar{E}}^{*}[I_\Es^{Col}]}
\\
    \node{N^{*}[I_U^{Col}]}
         \arrow{e,b}{i^*_{U, \gt'}}
    \node{M_{\gt'}^{*}[I_{\gt'}^{Col}]}
         \arrow{ne,t,3}{i^*_{\gt', \Es}}
         \arrow{e,b}{i^*_{\gt', \gt}}
    \node{M_{\tau}^{*}[I_\gt^{Col}]}
        \arrow[1]{n,b}{i^*_{\gt, \Es}}
\end{diagram}
\end{align*}
The forcing notion  $\mathbb{R}_{\bar{E}_{\kappa}}$ which we define later, adds a club to $\gk$.
For each $\gn_1, \gn_2$ successive points in the club
the cardinal structure and power function in the range $[\gn_1^+, \gn_2^{+3}]$
of the generic extension
is the same as the cardinal structure and power function in the range
$[\gk^+, j_\Es(\gk)^{+3}]$ of $M_{\bar{E}}^{*}[I_{\bar{E}}^{Col}]$.

\section{Redefining extender Sequences}
As in [4], in the prepared model $V=V^*[G]$  we define a new extender sequence system $\bar{F}= \langle \bar{F}_{\alpha}: \alpha \in \dom(\bar{F})\rangle$ by:
\begin{itemize}
  \item $\dom(\bar{F})=\dom(\bar{E}),$ \item $\len(\bar{F})=\len(\bar{E})$ \item $\leq_{\bar{F}}=\leq_{\bar{E}},$ \item $F(0)=E(0),$ \item $I(\tau)=I_{\tau},$ \item $\forall 0< \tau < \len(\bar{F}), F(\tau)= \langle \langle F_{\alpha}(\tau): \alpha \in \dom(\bar{F}) \rangle, \langle \pi_{\beta, \alpha}: \beta, \alpha \in \dom(\bar{F}), \beta \geq_{\bar{F}} \alpha \rangle \rangle$  is such that
    \begin{center}
    $X \in F_{\alpha}(\tau) \Leftrightarrow \langle \alpha, F(0), I(0), ..., F(\tau^{'}), I(\tau^{'}), ...: \tau^{'}  < \tau \rangle \in j_{\bar{E}}(X),$
    \end{center}
and
\begin{center}
$\pi_{\beta, \alpha}(\langle \xi, d \rangle)= \langle \pi_{\beta, \alpha}(\xi), d \rangle, $
\end{center}
\item $\forall \alpha \in \dom(\bar{F}), \bar{F}_{\alpha}= \langle \alpha, F(\tau),I(\tau): \tau < \len(\bar{F})  \rangle.$
\end{itemize}
Also let $I(\bar{F})$ be the filter generated by $\bigcup_{\tau < \len(\bar{F})} i_{\tau, \bar{E}}^{''}I(\tau).$ Then $I(\bar{F})$ is $\mathbb{R}_{\bar{E}}-$generic over $M_{\bar{E}}.$ Let us write $I(\bar{F})=I^{\Add}(\bar{F}) \times I^{\Col}(\bar{F})$ corresponding to $\mathbb{R}_{\bar{E}}=\mathbb{R}_{\bar{E}}^{\Add} \times \mathbb{R}_{\bar{E}}^{\Col}.$

From now on we work with this new definition of extender sequence
system and use $\bar{E}$ to denote it.

\begin{definition}$(1)$ We write $T \in \bar{E}_{\alpha}$ iff $\forall \xi < \len(\bar{E}_{\alpha}), T \in E_{\alpha}(\xi),$

$(2)$ $T \backslash \bar{\nu} = T \backslash V_{\kappa^{0}(\bar{\nu})}^{*},$

$(3)$ $T \upharpoonright \bar{\nu}= T \cap V_{\kappa^{0}(\bar{\nu})}^{*}.$

\end{definition}

We now define two forcing notions $\mathbb{P}_{\bar{E}}$ and $\mathbb{R}_{\bar{E}_{\kappa}}.$

\section{Definition of the forcing notion $\mathbb{P}_{\bar{E}}$}
 This forcing notion, defined in the ground model $V=V^*[G]$, is essentially the forcing notion of [4]. We give it in detail for completeness and later use. First we define a forcing notion $\mathbb{P}_{\bar{E}}^{*}.$

\begin{definition}
A condition $p$ in $\mathbb{P}_{\bar{E}}^{*}$ is of the form
\begin{center}
$p = \{ \langle \bar{\gamma}, p^{\bar{\gamma}}\rangle: \bar{\gamma} \in s \} \cup \{\langle \bar{E}_{\alpha}, T, f, F \rangle  \}$
\end{center}
where
\begin{enumerate}
\item $s \in [\bar{E}]^{\leq \kappa}, \min\bar{E}= \bar{E}_{\kappa} \in s,$

 \item$p^{\Es_{\kappa}} \in V_{\kappa^{0}(\bar{E})}^{*}$ is an extender sequence such that $\kappa(p^{\bar{E}_{\kappa}})$ is inaccessible ( we allow $p^{\bar{E}_{\kappa}}= \emptyset$). Write $p^0$ for $p^{\bar{E}_{\kappa}}.$

\item $\forall \bar{\gamma} \in s \backslash \{ \min(s) \}, p^{\bar{\gamma}} \in [V_{\kappa^{0}(\bar{E})}^{*}] ^{< \omega}$ is a $^{0}$-increasing sequence of extender sequences and $\max\kappa(p^{\bar{\gamma}})$ is inaccessible,

\item $\forall \bar{\gamma} \in s, \kappa(p^{0}) \leq \max\kappa(p^{\bar{\gamma}})$,

\item $\forall \bar{\gamma} \in s, \bar{E}_{\alpha} \geq \bar{\gamma},$

\item $T \in \bar{E}_{\alpha},$

\item $\forall \bar{\nu} \in T, \mid \{ \bar{\gamma} \in s: \bar{\nu}$ is permitted to $p^{\bar{\gamma}} \} \mid \leq \kappa^{0}(\bar{\nu}),$

\item $\forall \bar{\beta}, \bar{\gamma} \in s, \forall \bar{\nu} \in T,$ if $\bar{\beta} \neq \bar{\gamma}$ and $\bar{\nu}$ is permitted to $p^{\bar{\beta}}, p^{\bar{\gamma}},$ then $\pi_{\bar{E}_{\alpha}, \bar{\beta}}(\bar{\nu}) \neq \pi_{\bar{E}_{\alpha}, \bar{\gamma}}(\bar{\nu}),$

\item $f$ is a function such that

$\hspace{.5cm}$ $(9.1)$ $\dom(f)= \{\bar{\nu} \in T: \len(\bar{\nu})=0 \},$

$\hspace{.5cm}$ $(9.2)$ $f(\nu_{1}) \in R(\kappa(p^{0}), \nu_{1}^{0}).$ If $p^{0}=\emptyset,$ then $f(\nu_{1})= \emptyset,$

\item $F$ is a function such that

$\hspace{.5cm}$ $(10.1)$ $\dom(F)= \{ \langle \bar{\nu_{1}}, \bar{\nu_{2}} \rangle \in T^{2}:\len(\bar{\nu_{1}})=\len(\bar{\nu_{2}})= \emptyset \},$

$\hspace{.5cm}$ $(10.2)$ $F(\nu_{1}, \nu_{2}) \in R(\nu_{1}^{0}, \nu_{2}^{0}),$

$\hspace{.5cm}$ $(10.3)$ $j_{\bar{E}}^{2}(F)(\alpha, j_{\bar{E}}(\alpha)) \in I(\bar{E})$.
\end{enumerate}
\end{definition}
We write $\mc(p), \supp(p), T^{p}, f^{p}$ and $F^{p}$ for $\bar{E}_{\alpha}, s, T, f$ and $F$ respectively.
\begin{definition}
For $p, q \in \mathbb{P}_{\bar{E}}^{*},$ we say $p$ is a Prikry extension of $q$ ($p \leq^{*} q$ or $p \leq^{0} q$) iff
\begin{enumerate}
\item $\supp(p) \supseteq \supp(q),$

\item  $\forall \bar{\gamma} \in \supp(q), p^{\bar{\gamma}}=q^{\bar{\gamma}},$

\item  $\mc(p) \geq_{\bar{E}} \mc(q),$

\item  $\mc(p) >_{\bar{E}} \mc(q) \Rightarrow \mc(q) \in \supp(p),$

\item  $\forall \bar{\gamma} \in \supp(p) \backslash \supp(q), \max\kappa^{0}(p^{\bar{\gamma}}) > \bigcup \bigcup j_{\bar{E}}(f^{q})(\kappa(\mc(q))),$

\item  $T^{p} \leq \pi_{\mc(p), \mc(q)}^{-1''}T^{q},$

\item  $\forall \bar{\gamma} \in \supp(q), \forall \bar{\nu} \in T^{p},$ if $\bar{\nu}$ is permitted to $p^{\bar{\gamma}},$ then
\begin{center}
$\pi_{\mc(p), \bar{\gamma}}(\bar{\nu})=\pi_{\mc(q), \bar{\gamma}}(\pi_{\mc(p), \mc(q)}(\bar{\nu})),$
\end{center}

\item  $\forall \nu_{1} \in \dom(f^{p}), f^{p}(\nu_{1}) \leq f^{q}\circ\pi_{\mc(p), \mc(q)}(\nu_{1}),$

\item  $\forall \langle \nu_{1}, \nu_{2} \rangle \in \dom(F^{p}), F^{p}(\nu_{1}, \nu_{2}) \leq F^{q}\circ \pi_{\mc(p), \mc(q)}(\nu_{1}, \nu_{2}).$
\end{enumerate}

\end{definition}

We are now ready to define the forcing notion $\mathbb{P}_{\bar{E}}.$
\begin{definition}
A condition $p$ in $\mathbb{P}_{\bar{E}}$ is of the form
\begin{center}
$p=p_{n} ^{\frown} ...^{\frown} p_{0} $
\end{center}
where
\begin{itemize}
\item $p_{0} \in \mathbb{P}_{\bar{E}}^{*}, \kappa^{0}(p_{0}^{0}) \geq \kappa^{0}(\bar{\mu}_{1}),$ \item $p_{1} \in \mathbb{P}_{\bar{\mu}_{1}}^{*}, \kappa^{0}(p_{1}^{0}) \geq \kappa^{0}(\bar{\mu}_{2}),$

$\vdots$

    \item $p_{n} \in \mathbb{P}_{\bar{\mu}_{n}}^{*}.$
\end{itemize}
and $\langle \bar{\mu}_{n}, ..., \bar{\mu}_{1}, \bar{E} \rangle$ is a $^{0}-$inceasing sequence of extender sequence systems, that is $\kappa^{0}(\bar{\mu}_{n}) < ... < \kappa^{0}(\bar{\mu}_{1}) < \kappa^{0}(\bar{E}).$
\end{definition}
\begin{definition}
For $p, q \in \mathbb{P}_{\bar{E}},$ we say $p$ is a Prikry extension of $q$ ($p \leq^{*} q$ or $p \leq^{0} q$) iff
\begin{center}
$p=p_{n} ^{\frown} ...^{\frown} p_{0} $

$q=q_{n} ^{\frown} ...^{\frown} q_{0} $
\end{center}
where
\begin{itemize}
\item $p_{0}, q_{0} \in \mathbb{P}_{\bar{E}}^{*}, p_{0} \leq^{*} q_{0},$  \item $p_{1}, q_{1} \in \mathbb{P}_{\bar{\mu}_{1}}^{*}, p_{1} \leq^{*} q_{1},$

   $\vdots$

     \item $p_{n}, q_{n} \in \mathbb{P}_{\bar{\mu}_{n}}^{*}, p_{n} \leq^{*} q_{n}.$
\end{itemize}
\end{definition}
Now let $p \in \mathbb{P}_{\bar{E}}$ and $\bar{\nu} \in T^{p}.$ We define $p_{\langle \bar{\nu} \rangle}$ a one element extension of $p$ by $\bar{\nu}.$
\begin{definition}
Let $p \in \mathbb{P}_{\bar{E}}, \bar{\nu} \in T^{p}$ and $\kappa^{0}(\bar{\nu}) > \bigcup \bigcup j_{\bar{E}}(f^{p, \Col})(\kappa(\mc(p)))$, where $f^{p, \Col}$ is the collapsing part of $f^{p}$. Then $p_{\langle \bar{\nu}\rangle}=p_{1} ^{\frown} p_{0}$ where
\begin{enumerate}
\item $\supp(p_{0})=\supp(p),$

\item $\forall \bar{\gamma} \in \supp(p_{0}),$
$p_{0}^{\bar{\gamma}} = \left\{
\begin{array}{l}
      \pi_{\mc(p), \bar{\gamma}}(\bar{\nu}) \hspace{1.65cm} \text{ if } \bar{\nu} \text{ is permitted to } p^{\bar{\gamma}} \text{ and } \len(\bar{\nu}) >0, \\
       \pi_{\mc(p), \bar{\gamma}}(\bar{\nu}) \hspace{1.65cm} \text{ if } \bar{\nu} \text{ is permitted to } p^{\bar{\gamma}}, \len(\bar{\nu})=0 \text{ and } \bar{\gamma}=\bar{E}_{\kappa},
       \\
       p^{\bar{\gamma} \frown} \langle \pi_{\mc(p), \bar{\gamma}}(\bar{\nu}) \rangle \hspace{.7cm} \text{ if } \bar{\nu} \text{ is permitted to } p^{\bar{\gamma}}, \len(\bar{\nu})=0 \text{ and } \bar{\gamma}\neq \bar{E}_{\kappa},
       \\
       p^{\bar{\gamma}} \hspace{2.95cm} \text{ otherwise }.

     \end{array} \right.$

\item $\mc(p_{0})=\mc(p),$

\item $T^{p_{0}}=T^{p} \backslash \bar{\nu},$

\item $\forall \nu_{1} \in T^{p_{0}}, f^{p_{0}}(\nu_{1})=F^{p}(\kappa(\bar{\nu}), \nu_{1}),$

\item $F^{p_{0}}=F^{p},$

\item if $\len(\bar{\nu})>0$ then

$\hspace{.5cm}$ $(7.1)$ $\supp(p_{1})=\{\pi_{\mc(p), \bar{\gamma}}(\bar{\nu}): \bar{\gamma} \in \supp(p)$ and $\bar{\nu}$ is permitted to $p^{\bar{\gamma}}\},$

$\hspace{.5cm}$ $(7.2)$ $p_{1}^{\pi_{\mc(p), \bar{\gamma}}(\bar{\nu})}=p^{\bar{\gamma}},$

$\hspace{.5cm}$ $(7.3)$ $\mc(p_{1})=\bar{\nu},$

$\hspace{.5cm}$ $(7.4)$ $T^{p_{1}}=T^{p} \upharpoonright \bar{\nu},$

$\hspace{.5cm}$ $(7.5)$ $f^{p_{1}}=f^{p} \upharpoonright \bar{\nu},$

$\hspace{.5cm}$ $(7.6)$ $F^{p_{1}}=F^{p} \upharpoonright \bar{\nu},$

\item if $\len(\bar{\nu})=0$ then

$\hspace{.5cm}$ $(8.1)$ $\supp{p_{1}} =\{ \pi_{\mc(p),0}(\bar{\nu}) \},$

$\hspace{.5cm}$ $(8.2)$ $p_{1}^{\pi_{\mc(p),0}(\bar{\nu})}=p^{\bar{E}_{\kappa}},$

$\hspace{.5cm}$ $(8.3)$ $\mc(p_{1})=\bar{\nu}^{0},$

$\hspace{.5cm}$ $(8.4)$ $T^{p_{1}}= \emptyset,$

$\hspace{.5cm}$ $(8.5)$ $f^{p_{1}}=f^{p}(\kappa(\bar{\nu})),$

$\hspace{.5cm}$ $(8.6)$ $F^{p_{1}}= \emptyset.$
\end{enumerate}
\end{definition}
We use $(p_{\langle \bar{\nu} \rangle})_{0}$ and $(p_{\langle \bar{\nu} \rangle})_{1}$ for $p_{0}$ and $p_{1}$ respectively. We also let $p_{\langle \bar{\nu_{1}}, \bar{\nu_{2}} \rangle }= (p_{\langle \bar{\nu}_{1}\rangle})_{1} ^{\frown} (p_{\langle \bar{\nu}_{1} \rangle})_{0 \langle \bar{\nu_{2}} \rangle}$ and so on.

The above definition is the key step in the definition of the forcing relation $\leq.$
\begin{definition}
For $p, q \in \mathbb{P}_{\bar{E}},$ we say $p$ is a $1-$point extension of $q$ ($p \leq^{1} q$) iff
\begin{center}
$p=p_{n+1} ^{\frown} ...^{\frown} p_{0} $

$q=q_{n} ^{\frown} ...^{\frown} q_{0} $
\end{center}
and there is $0 \leq k \leq n$ such that
\begin{itemize}
\item $\forall i < k, p_{i}, q_{i} \in \mathbb{P}_{\bar{\mu}_{i}}^{*}, p_{i} \leq^{*} q_{i},$ \item $\exists \bar{\nu} \in T^{q_{k}}, (p_{k+1}) ^{\frown}p_{k} \leq^{*} (q_{k})_{ \langle \bar{\nu} \rangle}$ \item $\forall i > k, p_{i+1}, q_{i} \in \mathbb{P}_{\bar{\mu}_{i}}^{*}, p_{i+1} \leq^{*} q_{i},$
\end{itemize}
where $\bar{\mu}_{0}=\bar{E}.$
\end{definition}
\begin{definition}
For $p, q \in \mathbb{P}_{\bar{E}},$ we say $p$ is an $n-$point extension of $q$ ($p \leq^{n} q$) iff there are $p^{n}, ..., p^{0}$ such that
\begin{center}
$p=p^{n} \leq^{1} ... \leq^{1} p^{0}=q.$
\end{center}
\end{definition}
\begin{definition}
For $p, q \in \mathbb{P}_{\bar{E}},$ we say $p$ is an extension of $q$ ($p \leq q$) iff there is some $n$ such that $p \leq^{n} q$.
\end{definition}
Suppose that $H$ is $\mathbb{P}_{\bar{E}}-$generic over $V=V^*[G]$. For $\alpha \in \dom(\bar{E})$ set
\begin{center}
$C_{H}^{\alpha} = \{ \max\kappa(p_{0}^{\bar{E}_{\alpha}}): p \in H \}.$
\end{center}
\begin{theorem}
$(a)$ $V[H]$ and $V$ have the same cardinals $\geq \kappa,$

$(b)$ $\kappa$ remains strongly inaccessible in $V[H]$

$(c)$ $C_{H}^{\alpha}$ is unbounded in $\kappa,$

$(d)$ $C_{H}^{\kappa}$ is a club in $\kappa,$

$(e)$ $\alpha \neq \beta \Rightarrow C_{H}^{\alpha} \neq C_{H}^{\beta},$

$(f)$ Let $\l=\min(C_{H}^{\kappa}),$ and let $K$ be $\Col(\omega, \l^{+})_{V^[H]}-$generic over $V[H].$ Then
\begin{center}
$\C$$ARD^{V[H][K]} \cap\kappa = (\lim(C_{H}^{\kappa}) \cup \{\mu^{+}, ..., \mu^{+6}: \mu \in C_{H}^{\kappa}\} \backslash \l^{++}) \cup \{ \omega \},$
\end{center}

$(g)$ $V[H][K] \models  `` \forall \lambda \leq \kappa, 2^{\lambda}=\lambda^{+3}$''.
\end{theorem}
\begin{proof}
Essentially the same as in [4].
\end{proof}

\section{Definition of the forcing notion $\mathbb{R}_{\bar{E}_{\kappa}}$}
 We now define another forcing notion $\mathbb{R}_{\bar{E_{\kappa}}}.$ It is essentially the Radin forcing corresponding to $\bar{E}_{\kappa}$ with interleaving collapses (see also [3]).

\begin{definition}
A condition in $\mathbb{R}_{\bar{E}_{\kappa}}$ is of the form
\begin{center}
$p= \langle \langle \bar{\gamma}_{n}, s^{n}, S^{n}, f^{n}, F^{n}\rangle, ..., \langle \bar{\gamma}_{0}, s^{0}, S^{0}, f^{0}, F^{0} \rangle \rangle$
\end{center}
where
\begin{enumerate}
\item $\bar{\gamma}_{n}, ..., \bar{\gamma}_{0}$ are minimal extender sequences \footnote{An extender sequence $\bar{\gamma}$ is minimal, if it has length $1$ and $\kappa(\bar{\gamma})=\kappa^0(\bar{\gamma})$},

\item $\bar{\gamma}_{0}=\bar{E}_{\kappa},$

\item $\forall i \leq n-1, \kappa(\bar{\gamma}_{i+1}) < \kappa^{0}(\bar{\gamma}_{i}),$

\item $\forall i \leq n, S^{i} \in \bar{\gamma}_{i},$

\item $\forall i \leq n, s^{i} \in V_{\kappa^{0}(\bar{\gamma}_{i})}$ is a minimal extender sequence such that $\kappa(s^{i})$ is inaccessible,

\item $\forall i \leq n, f^{i}$ is a function such that

$\hspace{.5cm}$ $(6.1)$ $\dom(f^{i})=\{\bar{\nu} \in S^{i}: \len(\bar{\nu})=0 \},$

$\hspace{.5cm}$ $(6.2)$ $f^{i}(\nu_{1}) \in R^{\Col}(\kappa(s^{i}), \nu_{1}^{0}),$

\item $\forall i \leq n, F^{i}$ is a function such that

$\hspace{.5cm}$ $(7.1)$ $\dom(F^{i})=\{ \langle \bar{\nu_{1}}, \bar{\nu_{2}} \rangle \in (S^{i})^{2}: \len(\bar{\nu}_{1})=\len(\bar{\nu}_{2})=0 \},$

$\hspace{.5cm}$ $(7.2)$ $F^{i}( \langle \nu_{1}, \nu_{2} \rangle) \in R^{\Col}(\nu_{1}^{0}, \nu_{2}^{0}),$

$\hspace{.5cm}$ $(7.3)$ $j_{\bar{E}}^{2}(F^{i})(\kappa(\bar{\gamma}_{i}), j_{\bar{E}}(\kappa(\bar{\gamma}_{i})) \in I^{\Col}(\bar{E}).$
\end{enumerate}
\end{definition}
\begin{definition}
For $p, q \in \mathbb{R}_{\bar{E}_{\kappa}}$ we say $p$ is a Prikry extension of $q$ ($p \leq^{*} q$ or $p \leq^{0} q$) iff $p$ and $q$ are of the form
\begin{center}
$p= \langle \langle \bar{\gamma}_{n}, s^{n}, S^{n}, f^{n}, F^{n}\rangle, ..., \langle \bar{\gamma}_{0}, s^{0}, S^{0}, f^{0}, F^{0} \rangle \rangle$

$q= \langle \langle \bar{\gamma}_{n}, t^{n}, T^{n}, g^{n}, G^{n}\rangle, ..., \langle \bar{\gamma}_{0}, t^{0}, T^{0}, g^{0}, G^{0} \rangle \rangle$
\end{center}
where $\forall i \leq n$
\begin{enumerate}
\item $s^{i}=t^{i},$

\item $S^{i} \subseteq T^{i},$

\item $f^{i} \leq g^{i},$

\item $F^{i} \leq G^{i}.$
\end{enumerate}
\end{definition}
\begin{definition}
Let $p= \langle \langle \bar{\gamma}_{n}, s^{n}, S^{n}, f^{n}, F^{n}\rangle, ..., \langle \bar{\gamma}_{0}, s^{0}, S^{0}, f^{0}, F^{0} \rangle \rangle \in \mathbb{R}_{\bar{E}_{\kappa}},$ and let $ \langle \bar{\nu} \rangle \in S^{i}, \kappa^{0}(\bar{\nu}) > \bigcup \bigcup j_{\bar{E}}(f^{i})(\kappa(\bar{\gamma}_{i})).$ We define $p_{\langle \bar{\nu} \rangle}$ as follows
\begin{itemize}
\item if $\len(\bar{\nu})> 0,$ then

$p_{\langle \bar{\nu} \rangle}= \langle \langle \bar{\gamma}_{n}, s^{n}, S^{n}, f^{n}, F^{n}\rangle, ...,$

$\hspace{2cm}$  $\langle \bar{\gamma}_{i+1}, s^{i+1}, S^{i+1}, f^{i+1}, F^{i+1}\rangle,$

$\hspace{3cm}$  $\langle \bar{\nu}, s^{i}, S^{i} \upharpoonright\bar{\nu}, f^{i} \upharpoonright \bar{\nu}, F^{i} \upharpoonright \bar{\nu}\rangle,$

 $\hspace{4cm}$ $\langle \bar{\gamma}_{i}, \bar{\nu}, S^{i} \backslash \bar{\nu}, F^{i}(\kappa(\bar{\nu},-)), F^{i}\rangle,$

 $\hspace{5cm}$  $\langle \bar{\gamma}_{i-1}, s^{i-1}, S^{i-1}, f^{i-1}, F^{i-1}\rangle, ...,$

 $\hspace{6cm}$ $\langle \bar{\gamma}_{0}, s^{0}, S^{0}, f^{0}, F^{0}\rangle \rangle $
 \item if $\len(\bar{\nu})= 0,$ then

$p_{\langle \bar{\nu} \rangle}= \langle \langle \bar{\gamma}_{n}, s^{n}, S^{n}, f^{n}, F^{n}\rangle, ...,$

$\hspace{2cm}$  $\langle \bar{\gamma}_{i+1}, s^{i+1}, S^{i+1}, f^{i+1}, F^{i+1}\rangle,$

$\hspace{3cm}$  $\langle \bar{\nu}^{0}, s^{i}, \emptyset, f^{i}(\kappa(\bar{\nu})), \emptyset \rangle,$

 $\hspace{4cm}$ $\langle \bar{\gamma}_{i}, \bar{\nu}, S^{i} \backslash \bar{\nu}, F^{i}(\kappa(\bar{\nu},-)), F^{i}\rangle,$

 $\hspace{5cm}$  $\langle \bar{\gamma}_{i-1}, s^{i-1}, S^{i-1}, f^{i-1}, F^{i-1}\rangle, ...,$

 $\hspace{6cm}$ $\langle \bar{\gamma}_{0}, s^{0}, S^{0}, f^{0}, F^{0}\rangle \rangle $
\end{itemize}
\end{definition}
\begin{definition}
Let $p, q \in \mathbb{R}_{\bar{E}_{\kappa}},$ where $q= \langle \langle \bar{\gamma}_{n}, s^{n}, S^{n}, f^{n}, F^{n}\rangle, ..., \langle \bar{\gamma}_{0}, s^{0}, S^{0}, f^{0}, F^{0} \rangle \rangle $. We say $p$ is a $1-$point extension of $q$ ($p \leq^{1} q$) iff there are $i$ and $\langle \bar{\nu} \rangle \in S^{i}$ such that $p \leq^{*} q_{\langle \bar{\nu} \rangle}.$
\end{definition}
\begin{definition}
Let $p, q \in \mathbb{R}_{\bar{E}_{\kappa}}.$  We say $p$ is an $n-$point extension of $q$ ($p \leq^{n} q$) iff there are $p^{n}, ..., p^{0}$ such that
\begin{center}
$p=p^{n} \leq^{1} ... \leq^{1} p^{0}=q.$
\end{center}
\end{definition}
\begin{definition}
Let $p, q \in \mathbb{R}_{\bar{E}_{\kappa}}.$  We say $p$ is an extension of $q$ ($p \leq q$) iff there is  $n$ such that $p \leq^{n} q$.
\end{definition}

Suppose $G$ is $\mathbb{R}_{\bar{E}_{\kappa}}-$generic over $V$. Set
\begin{center}
$C=\{\kappa(s^{0}): s^{0}$ appears in in some $p \in G \}.$
\end{center}
\begin{theorem}
$(a)$ $V[G]$ and $V$ have the same cardinals $\geq \kappa,$

$(b)$ $\kappa$ remains strongly inaccessible in $V[G],$

$(c)$ $C$ is a club in $\kappa,$

$(d)$ Let $\l=\min(C)$ and let $K$ be $\Col(\omega, \l^{+})_{V[G]}-$generic over $V[G].$ Then
\begin{center}
$\C$$ARD^{V[G][K]} \cap\kappa = (\lim(C) \cup \{\gamma^{+}, ..., \gamma^{+6}: \gamma \in C \} \backslash \l^{++}) \cup \{ \omega \},$
\end{center}
$(e)$ $V[G][K] \models `` GCH$''.
\end{theorem}
\begin{proof}
Essentially the same as in [3] and [4].
\end{proof}

\section{Projection of $\mathbb{P}_{\bar{E}}$ into $\mathbb{R}_{\bar{E}_{\kappa}}$}
 We now define a projection
\begin{center}
$\pi: \mathbb{P}_{\bar{E}} \rightarrow \mathbb{R}_{\bar{E}_{\kappa}}$
\end{center}
as follows. 
Suppose $p=p_{n} ^{\frown} ...^{\frown} p_{0} $
where
\begin{itemize}
\item $p_{0} \in \mathbb{P}_{\bar{E}}^{*}, \kappa^{0}(p_{0}^{0}) \geq \kappa^{0}(\bar{\mu}_{1}),$ \item $p_{1} \in \mathbb{P}_{\bar{\mu}_{1}}^{*}, \kappa^{0}(p_{1}^{0}) \geq \kappa^{0}(\bar{\mu}_{2}),$

  $\vdots$

    \item $p_{n} \in \mathbb{P}_{\bar{\mu}_{n}}^{*}.$
\end{itemize}
and $\langle \bar{\mu}_{n}, ..., \bar{\mu}_{1}, \bar{\mu}_{0} \rangle$, where $\bar{\mu}_{0}= \bar{E},$ is a $^{0}-$inceasing sequence of extender sequence systems. For each $i \leq n$ set $f^{p_{i}}=f^{p_{i}, \Add} \times f^{p_{i}, \Col}$ and $F^{p_{i}}=F^{p_{i}, \Add} \times F^{p_{i}, \Col}$ which correspond  to $R=R^{\Add} \times R^{\Col}.$ Given $p$ as above, for each $i\leq n,$ we have
\begin{enumerate}
\item $j(f^{p_i, \Col})(\kappa(mc(p_i)))\in V_{\kappa^0(mc(p_i))}.$ Hence there is a function $g^{p_i}$ such that 
 \begin{center}
 $j(f^{p_i, \Col})(\kappa(mc(p_i)))=j(g^{p_i})(\kappa^0(mc(p_i))),$
\end{center}
\item $j_2(F^{p_i, \Col})(\kappa(mc(p_i)), j(\kappa(mc(p_i))))$ is in the generic filter constructed through the normal measure. Hence there is a stronger condition in the filter which is the image of a condition from the generic over the normal ultrapower. I.e there is a function $H^{p_i}$ such that 
    \begin{center}
    $j_2(H^{p_i})(\kappa^0(mc(p_i)), j(\kappa^0(mc(p_i)))) \leq j_2(F^{p_i, \Col})(\kappa(mc(p_i)), j(\kappa(mc(p_i)))),$ 
    \end{center}
 and   
     \begin{center}
    $j_2(H^{p_i})(\kappa^0(mc(p_i)), j(\kappa^0(mc(p_i))))\in I^{\Col}$ 
    \end{center}    
    and there is no weaker function $H'$ satisfying this.

\end{enumerate}
Let $(T^{p_{i}})^*$ be obtained from $T^{p_{i}}$ by replacing extender sequences in $T^{p_{i}}$ of length $0$ with 
\begin{center}
$\{ \nu \in T^{p_i}: \len(\nu)=0,  f^{p_i, \Col}(\nu)=g^{p_i}(\pi_{\mc(p_{i}),0}(\nu))   \}.$
\end{center}
It follows from $(1)$ that $(T^{p_{i}})^* \in mc(p_i).$ Now let $(T^{p_{i}})^{**}$ be obtained from $(T^{p_{i}})^{*}$  by restricting extender sequences in $(T^{p_{i}})^{*}$ of length $0$ to those $\nu_1 \in (T^{p_i})^*, \len(\nu_1)=0,$ such that
\begin{center}
$ \{\nu_2 \in (T^{p_i})^* : \len(\nu_2)=0, H^{p_i}(\pi_{\mc(p_{i}),0}(\nu_1), \pi_{\mc(p_{i}),0}(\nu_2)) \leq  F^{p_i, \Col}(\nu_1, \nu_2) \}$
\end{center}
has measure one with respect to the normal measure determined by $mc(p_i).$ Then by $(2)$ we have $(T^{p_{i}})^{**} \in mc(p_i).$
Let
\begin{center}
$\pi(p)= \langle  \langle \min\bar{\mu}_{n}, p_{n}^{0}, A^{p_n}, g^{p_{n}}, H^{p_{n}} \rangle, ..., \langle \bar{E}_{\kappa}, p_{0}^{0}, A^{p_0}, g^{p_{0}}, H^{p_{0}} \rangle \rangle,$
\end{center}
where $A^{p_i} =\{\pi_{\mc(p_{i}),0}(\bar{\nu}): \bar{\nu} \in (T^{p_{i}})^{**} \}$.
Let us note that $\pi(p) \in \mathbb{R}_{\bar{E}_{\kappa}}$ and $\pi$ is well-defined.

\begin{lemma}
$\pi$ is a projection, i.e

$(a)$ $\pi(1_{\mathbb{P}_{\bar{E}}})=1_{\mathbb{R}_{\bar{E}_{\kappa}}},$

$(b)$ $\pi$ is order preserving,

$(c)$ if $p \in \mathbb{P}_{\bar{E}}, q \in \mathbb{R}_{\bar{E}_{\kappa}}$ and $q \leq \pi(p)$ then there is $r \leq p$ in $\mathbb{P}_{\bar{E}}$ such that $\pi(r) \leq q.$
\end{lemma}
\begin{proof}
Parts $(a)$ and $(b)$ are trivial; let us prove $(c)$. Let $p\in \mathbb{P}_{\bar{E}}$, $q \in \mathbb{R}_{\bar{E}_{\kappa}},$ and suppose that $q \leq \pi(p).$ Let us suppose for simplicity that $p\in \mathbb{P}^*_{\bar{E}}$\footnote{In fact the general case follows from this special case using the factorization properties of  $\mathbb{P}_{\bar{E}}$.}.  Let $\pi(p)=\langle  \langle \bar{E}_{\kappa}, p^{0}, A^p, g^{p}, H^{p} \rangle \rangle.$ Since $q \leq \pi(p)$, there is some $k$ such that  $q \leq^k \pi(p).$ We prove the lemma by induction on $k.$

First suppose that $k=0,$ so that $q$ is a Prikry extension of $\pi(p).$ Let $q= \langle \langle \bar{E}_{\kappa}, t, T, g, G\rangle \rangle.$ Then we have $t=p^{0}, T \subseteq A^p, g \leq g^{p}$ and $ G \leq H^{p}.$ Let $r\in \mathbb{P}^*_{\bar{E}}, r\leq^* p$ be obtained from $p$ with the following changes:
\begin{itemize}
\item $T^r \subseteq \{\bar{\nu} \in (T^p)^{**}: \pi_{mc(p),0}(\bar{\nu})\in T  \}, T^r \in mc(p),$
\item $f^{r, \Col} \leq f^{p, \Col}$ is such that for all $\nu \in T^r$ of length $0$
\begin{center}
$ g^r(\pi_{mc(p),0}(\nu))\leq g(\pi_{mc(p),0}(\nu)),$
\end{center}
\item $F^{r, \Col} \leq F^{p, \Col}$ is such that for all $\langle \nu_1, \nu_2 \rangle \in (T^r)^2$ with $\len(\nu_1)=\len(\nu_2)=0,$
\begin{center}
$ H^r(\pi_{mc(p),0}(\nu_1), \pi_{mc(p),0}(\nu_2))\leq G(\pi_{mc(p),0}(\nu_1), \pi_{mc(p),0}(\nu_2)).$
\end{center}
\end{itemize}
Then $\pi(r)=\langle  \langle \bar{E}_{\kappa}, p^{0}, A^r, g^{r}, H^{r} \rangle \rangle$, where $A^r   \subseteq T$ and for all $\nu, \nu_1$ and $\nu_2$ such that their image under $\pi_{\mc(p),0}$ is in $A^r$, we have
\begin{center}
$g^r(\pi_{mc(p),0}(\nu))\leq g(\pi_{mc(p),0}(\nu))$
\end{center}
and
\begin{center}
$H^r(\pi_{mc(p),0}(\nu_1), \pi_{\mc(p),0}(\nu_2))\leq G(\pi_{mc(p),0}(\nu_1), \pi_{\mc(p),0}(\nu_2))$.
\end{center}
It follows that $\pi(r)\leq^* q.$

Now suppose that $k=1$ (the general case $k\geq 1$ can be proved similarly). Let $\bar{\nu} \in A^p$ be such that $q \leq^* (\pi(p))_{\langle\bar{\nu}\rangle}.$
Also let $q= \langle \langle \bar{\gamma}_{1}, t^{1}, T^{1}, g^{1}, G^{1}\rangle, \langle \bar{\gamma}_{0}, t^{0}, T^{0}, g^{0}, G^{0} \rangle \rangle.$

Suppose for example that
 $l(\bar{\nu})>0$ (the case $l(\bar{\nu})=0$ can be proved similarly). Then $(\pi(p))_{\langle \bar{\nu}\rangle}= \langle \langle \bar{\nu}, p^0, A^p \upharpoonright \bar{\nu}, g^p \upharpoonright \bar{\nu}, H^p \upharpoonright \bar{\nu}  \rangle, \langle \bar{E}_{\kappa}, \bar{\nu}, A^p \setminus \bar{\nu}, H^p(\kappa(\bar{\nu},-)), H^p     \rangle \rangle.$ Let $\bar{\mu} \in T^p$ be an extender sequence system of size $(\kappa^0(\bar{\nu}))^{+3},$ obtained by the same elementary embedding generating $\bar{\nu},$ such that $\min \bar{\mu}=\bar{\nu}$ and let $(p)_{\langle \bar{\mu} \rangle}=p_1^{\frown} p_0.$ Let $r=r_1^{\frown} r_0\in \mathbb{P}_{\bar{E}}, r \leq^* (p)_{\langle \bar{\mu} \rangle}$ be such that for $i\in\{ 0,1\}$:
\begin{itemize}
\item $\supp(r_i)=\supp(p_i),$
\item $T^{r_i} \subseteq \{\bar{\nu} \in (T^{p_i})^{**}: \pi_{mc(p_i),0}(\bar{\nu})\in T^i  \}, T^{r_i} \in mc(p_i),$
\item $f^{r_i, \Col} \leq f^{p_i, \Col}$ is such that for all $\nu \in T^{r_i}$ with $\len(\nu)=0,$
\begin{center}
$ g^{r_i}(\pi_{mc(p_i),0}(\nu))\leq g^i(\pi_{mc(p_i),0}(\nu)),$
\end{center}
\item $F^{r_i, \Col} \leq F^{p, \Col}$ is such that for all $\langle \nu_1, \nu_2 \rangle \in (T^{r_i})^2$ with $\len(\nu_1)=\len(\nu_2)=0$,
\begin{center}
$ H^{r_i}(\pi_{mc(p_i),0}(\nu_1), \pi_{mc(p_i),0}(\nu_2))\leq G^i(\pi_{mc(p_i),0}(\nu_1), \pi_{mc(p_i),0}(\nu_2)).$
\end{center}
\end{itemize}
Then  $\pi(r)=\pi(r_1)^{\frown} \pi(r_0)$ and as above $ \pi(r_i) \leq^* \langle \langle \bar{\gamma}_{i}, t^{i}, T^{i}, g^{i}, G^{i}\rangle \rangle, i\in\{0,1\}.$ It follows that $\pi(r)\leq^* q.$

The lemma follows.
\end{proof}
\section{Completing the proof}
Finally in this section we complete the proof of Theorem 1.1. Let $H$ be $\mathbb{P}_{\bar{E}}-$generic over $V$ and let
$H_0=\langle \pi^{''}H  \rangle,$ the filter generated by $\pi^{''}H.$
Then $H_0$ is $\mathbb{R}_{\bar{E}_{\kappa}}-$generic over $V.$
 Consider the clubs
$C=\{\kappa(s^{0}): s^{0}$ appears in in some $p \in H_0 \}$ and
$C_{H}^{\kappa}=\{\kappa(p_{0}^{0}):p \in H \}.$ It is easily seen that $C=C_{H}^{\kappa}.$ Let $\l=\min(C).$ Note that the forcing notions $\mathbb{P}_{\bar{E}}$ and $\mathbb{R}_{\bar{E}_{\kappa}}$ add no new bounded subsets to $\l^{+},$ hence $\Col(\omega, \l^{+})_{V[H_0]}=\Col(\omega, \l^{+})_{V[H]},$ and hence if $K$ is $\Col(\omega, \l^{+})_{V[H]}-$generic over $V[H]$ then $K$ is $\Col(\omega, \l^{+})_{V[H_0]}-$generic over $V[H_0].$
Let
\begin{center}
$V_{1}=V_{\kappa}^{V[H_0][K]}$

$V_{2}=V_{\kappa}^{V[H][K]}$
\end{center}
It follows that $V_1$ and $V_2$ are models of $ZFC$. We show that the pair $(V_1, V_2)$ satisfies the requirements of the theorem.

$(a)$ $V_1$ and  $V_2$ have the same cardinals:
This is trivial, since

 $\hspace{2cm}$$\C$$ARD^{V_{1}} = (\lim(C) \cup \{\mu^{+}, ..., \mu^{+6}: \mu \in C \} \backslash \l^{++}) \cup \{ \omega \}$

 $\hspace{3.55cm}$$= (\lim(C_{H}^{\kappa}) \cup \{\mu^{+}, ..., \mu^{+6}: \mu \in C_{H}^{\kappa}\} \backslash \l^{++}) \cup \{ \omega \}.$

 $\hspace{3.55cm}$$=\C$$ARD^{V_2}.$

$(b)$ $V_1$ and  $V_2$ have the same cofinalities:
This is again trivial, since changing the cofinalities depends on the length of the extender sequence system used and not on its size.

$(c)$ $V_{1} \models `` GCH$'': by Theorem 6.7$(e)$.

$(d)$ $V_{2} \models `` \forall \lambda, 2^{\lambda}=\lambda^{+3}$': by Theorem 5.9$(g)$.

Theorem 1.1 follows.

\emph{Open question.} Is it possible to kill $GCH$ everywhere,
preserving cofinalities, adding just a single real? (Allowing
 cofinalities, but not cardinalities, to change, this was accomplished in [1]).

{Kurt G\"{o}del Research Center for Mathematical Logic, University of Vienna,

E-mail address: sdf@logic.univie.ac.at}

{School of Mathematics, Institute for Research in
Fundamental Sciences (IPM), P.O. Box: 19395-5746, Tehran-Iran.

E-mail address: golshani.m@gmail.com}


\begin{thebibliography}{xx}


\bibitem{ } Friedman, S. and Golshani, M., Killing the $GCH$ everywhere with a single real, J. Symbolic Logic 78 (2013), no 3, 803--823.

\bibitem{ } Gitik, M. Prikry-type forcings. Handbook of set theory. Vols. 1, 2, 3, 1351–-1447, Springer, Dordrecht, 2010.

\bibitem{ } Merimovich, C. Extender-based Radin forcing. Trans. Amer. Math. Soc. 355 (2003), no. 5, 1729–-1772.


\bibitem{ } Merimovich, C. A power function with a fixed finite gap everywhere. J. Symbolic Logic 72 (2007), no. 2, 361–-417.


\end{thebibliography}
\end{document}